\documentclass[11pt]{amsart}
\usepackage{tikz}
\usetikzlibrary{arrows.meta, decorations.pathmorphing}
\usepackage{amssymb}
\usepackage{amsmath, amsfonts, vmargin}
\usepackage{color}
\usepackage{amsthm}
\usepackage[colorlinks=true,citecolor=blue,linkcolor=blue]{hyperref}
\IfFileExists{bbm.sty}{\usepackage{bbm}}{}
\usepackage{verbatim}
\usepackage{mathrsfs}
\usepackage{enumerate}
\usepackage{cite}
\usepackage{array}
\usepackage{multirow}
\usepackage{booktabs}
\usepackage{mathtools}
\mathtoolsset{showonlyrefs=true}
\newtheorem{theorem}{Theorem}
\newtheorem{proposition}[theorem]{Proposition}

\newtheorem{corollary}[theorem]{Corollary}
\newtheorem{lemma}[theorem]{Lemma}

\newtheorem{remark}[theorem]{Remark}
\newtheorem{definition}[theorem]{Definition}
\theoremstyle{definition}

\numberwithin{theorem}{section}
\numberwithin{equation}{section}

\def\exp{\operatorname{\rm exp}}

\def\Tr{{\rm Tr}}

\newcommand{\Norm}[1]{ \left\|  #1 \right\| }

\newcommand{\inner}[1]{ \langle  #1 \rangle}

\newcommand{\dint}{\mathchoice{{\vbox{\hbox{$\textstyle-$}}\kern-7.5pt}}
    {{\vbox{\hbox{$\scriptstyle-$}}\kern-5.25pt}}
    {{\vbox{\hbox{$\scriptscriptstyle-$}}\kern-3.75pt}}
    {{\vbox{\hbox{$\scriptscriptstyle-$}}\kern-3.0pt}}\!\int}
\newcolumntype{C}[1]{>{\centering\arraybackslash}p{#1}}

\def\RR{{\mathbb R}}
\def\TT{{\mathbb T}}

\def\S{\mathcal{S}}

\def\M{{\mathcal M}}

\begin{document}

\title{Cwikel estimates and Weyl's laws for $L\log L$ potentials \\ on quantum tori}

\author{Siyu Liu}
\address{Institute for Advanced Study in Mathematics, Harbin Institute of Technology, 150001 Harbin, China}
\email{syliu@stu.hit.edu.cn}

\author{Xinyu Zhang}
\address{Institute for Advanced Study in Mathematics, Harbin Institute of Technology, 150001 Harbin, China}
\email{xy.zhang@stu.hit.edu.cn}


\begin{abstract}
We prove an \(L\log L\)-endpoint Cwikel estimate on quantum tori $\mathbb T_\theta^d$, i.e., for every $d\geq 2$ and $x\in L\log L(\mathbb T_\theta^d)$,
\begin{equation*}
    \big\|(1-\Delta)^{-\frac{d}{4}}\lambda(x)(1-\Delta)^{-\frac{d}{4}}\big\|_{\mathcal{L}_{1,\infty}} \leq C_d \Norm{x}_{L\log L(\TT_\theta^d)}, 
\end{equation*} 
where $\lambda(x)$ is the left multiplication by $x$. 
Our argument actually gives an abstract weak Hilbert--Schmidt estimate for multiplication operators in finite von Neumann algebras whose $L_2$-spaces admit orthonormal bases of unitaries. 
As a consequence, it yields a critical Sobolev embedding $W_2^{d/2}\hookrightarrow\exp(L_2)$ on quantum tori.
When $\theta=0$, our method also gives a simpler proof of the $L\log L$ Cwikel estimate on classical tori.

As an application, we use the $C^*$-algebraic principal symbol map
to prove Weyl's laws for sandwich operators 
\[
 S(1-\Delta)^{-d/4}\lambda(x)(1-\Delta)^{-d/4}T,
\]
where $x\in L\log L(\mathbb T_\theta^d)$ and
$S,T\in\Pi$, the $C^*$-algebra generated by
$\lambda(C(\mathbb T_\theta^d))$ and Fourier multipliers
with continuous homogeneous symbols of degree zero.
We also prove a semiclassical Weyl's law for critical Schr\"odinger operators with $L\log L$ potentials.
\end{abstract}

\maketitle

\markboth{S. Liu and X. Zhang}{Cwikel estimates and Weyl's laws for $L\log L$ potentials on quantum tori}

\tableofcontents

\section{Introduction}

\subsection{Cwikel estimates and the classical endpoint}

Estimates for products of multiplication operators and Fourier
multipliers originate in the study of bound states of Schr\"odinger
operators. Through the Birman--Schwinger principle, they relate the
integrability of a potential to the number of negative eigenvalues.
A basic question, formulated by Simon \cite{Simon76}, is to determine
when $M_fg(-\mathrm i\nabla)$ belongs to a weak Schatten class, where
$M_f$ denotes multiplication by $f$ and $g(-\mathrm i\nabla)$ is the
Fourier multiplier with symbol $g$. We write $\mathcal{L}_{p,\infty}$
for the weak Schatten ideal, with
$\|T\|_{p,\infty}=\sup_{j\geq0}(j+1)^{1/p}\mu(j,T)$, where
$\mu(j,T)$ are the singular values in nonincreasing order.
Cwikel's theorem \cite{Cwikel77} gives
\begin{equation}
\label{eq:intro-cwikel}
 \|M_fg(-\mathrm i\nabla)\|_{p,\infty}
 \leq C_p\|f\|_{L_p(\mathbb R^d)}
          \|g\|_{L_{p,\infty}(\mathbb R^d)},\qquad 2<p<\infty,
\end{equation}
where the space $L_{p,\infty}(\RR^d)$ is the usual weak $L_p$-space.
Subsequent developments include Weidl's estimates in terms of
$f\otimes g$ \cite{Weidl99} and the noncommutative framework of
Levitina, Sukochev and Zanin \cite{LSZ20}, with applications to
noncommutative Euclidean spaces and magnetic Laplacians.

For the Bessel potential symbol
$g(\xi)=(1+|\xi|^2)^{-d/(2p)}$, the exponent $p=2$ is critical:
estimate~\eqref{eq:intro-cwikel} fails at this endpoint under
the same integrability assumptions. On compact spaces, the relevant substitute
is the Orlicz space $L\log L$, equipped here with the Luxemburg norm
associated with $\Phi(t)=t\log(e+t)$, $t\geq0$.
Building on the approximation methods of Birman and Solomyak
\cite{BS80}, Solomyak \cite{Sol94,Sol95} obtained critical estimates
in even dimensions. Sukochev and Zanin \cite[Theorem~1.1]{SZ22}
extended these estimates to every dimension: for $d\geq1$, on the
classical torus $\mathbb T^d=\mathbb R^d/(2\pi\mathbb Z)^d$ equipped
with normalized Haar measure $\nu$ and nonpositive flat Laplacian
$\Delta$, they proved
\begin{equation}
\label{eq:intro-classical-endpoint}
 \big\|(1-\Delta)^{-\frac{d}{4}}M_f(1-\Delta)^{-\frac{d}{4}}\big\|_{1,\infty}
 \leq C_d\|f\|_{L\log L(\mathbb T^d)},
 \qquad f\in L\log L(\mathbb T^d).
\end{equation}
Here $C_d$ depends only on $d$. 
One should note that even the operator norm bound
\begin{equation}
\label{eq:intro-classical-operator-norm}
 \bigl\|(1-\Delta)^{-\frac{d}{4}}M_f(1-\Delta)^{-\frac{d}{4}}\bigr\|_\infty
 \leq C_d\|f\|_{L\log L(\mathbb T^d)}
\end{equation}
is nontrivial. 
The estimate \eqref{eq:intro-classical-operator-norm} follows from the critical Sobolev embedding
$W_2^{d/2}(\mathbb T^d)\hookrightarrow\exp(L_2)(\mathbb T^d)$
and H\"older's inequality for Orlicz spaces, see \cite[Section~1.1]{SZ22}.
The space $L\log L$ is the largest Orlicz space for which
the operator norm estimate \eqref{eq:intro-classical-operator-norm} holds;
see \cite[Corollary~24]{SZ23Optimal}.
Consequently, the $L\log L$ assumption in the weak trace class
estimate \eqref{eq:intro-classical-endpoint} cannot be weakened within the class of Orlicz spaces.

The classical proof of \eqref{eq:intro-classical-endpoint} combines the
critical Sobolev embedding above with Solomyak's finite-rank
approximation method \cite[Section~1.1]{SZ22}.
More precisely, for $f\geq0$, \eqref{eq:intro-classical-endpoint} is equivalent to the following $\mathcal{L}_{2,\infty}$ estimate 
\begin{equation}
      \label{eq:intro-classical-endpoint-2-infty}
      \big\|M_{f^{1/2}}(1-\Delta)^{-\frac{d}{4}}\big\|_{2,\infty}^2\leq C_d\|f\|_{L\log L(\mathbb{T}^d)}.
\end{equation}
Since $(1-\Delta)^{-d/4}$ maps $L_2(\TT^d)$ onto $W_2^{d/2}(\TT^d)$,
the inequality \eqref{eq:intro-classical-endpoint-2-infty} controls the approximation numbers
of multiplication by $f^{1/2}$ from $W_2^{d/2}$ into $L_2$.
In\cite{SZ22}, the authors construct (for $n\geq1$) $K_n$ of rank at
most $C_dn$ such that
\begin{equation}
\label{eq:Solomyak approximation}
 \int_{\mathbb T^d}f|u-K_nu|^2\,d\nu
 \leq\frac{C_d}{n}\|f\|_{L\log L}\|u\|_{W_2^{d/2}}^2,
 \qquad u\in W_2^{\frac{d}{2}}(\mathbb{T}^d),\quad f\geq0.
\end{equation}
Together with \eqref{eq:intro-classical-operator-norm},
estimate \eqref{eq:Solomyak approximation} yields
\eqref{eq:intro-classical-endpoint-2-infty}.
The construction of such $K_n$'s uses local
polynomial approximation and a covering by toric cubes selected
through Besicovitch's covering lemma (see \cite[Lemma 3.4]{SZ22}). 
These spatial tools have no direct counterpart for operator-valued potentials on a quantum torus.
This motivates an approach to the noncommutative endpoint estimate
that avoids spatial localization.

\subsection{Noncommutative integration and Connes' formula}

The Cwikel estimates also play a role in Connes' noncommutative integration
\cite{AC88,AC94,LSZ21}. Let $H$ be a separable infinite-dimensional
Hilbert space and $\mathcal B(H)$ its algebra of bounded operators.
Given a positive operator $T\in\mathcal{L}_{1,\infty}(H)$ and a positive
normalized trace $\varphi$ on this ideal, a noncommutative integral
has the form
\[
 A\longmapsto\varphi(AT),\qquad A\in\mathcal B(H),
\]
where $T$ plays the role of a volume element. Such traces
vanish on finite-rank operators; the normalization assigns value one
to an operator with eigenvalues $(j+1)^{-1}$, $j\geq0$.
Connes' trace theorem \cite[Theorem~1]{AC88} identifies the Dixmier trace of a classical pseudodifferential operator of order
$-d$ with $1/d$ times its noncommutative residue. Let $(X,g)$ be a closed $d$-dimensional Riemannian manifold,
with Riemannian volume $\operatorname{vol}_g$ and nonnegative
Laplace--Beltrami operator $-\Delta_g$. Write $|\mathbb S^{d-1}|$
for the surface area of the unit sphere in $\mathbb R^d$.
Connes' integration formula then reads
\begin{equation}
\label{eq:intro-classical-integration}
 \varphi\bigl((1-\Delta_g)^{-\frac{d}{4}}M_f(1-\Delta_g)^{-\frac{d}{4}}\bigr)
 =\frac{|\mathbb S^{d-1}|}{d(2\pi)^d}
       \int_X f\,d\operatorname{vol}_g,
 \qquad f\in C^\infty(X).
\end{equation}

Connes asked whether ordinary eigenvalue asymptotics could give
this integration formula without extended limits; see
\cite[Question~1.1]{SZ23}.
For a selfadjoint operator or real-valued function $b$, write
$b_\pm=(|b|\pm b)/2$ for its positive and negative parts.
For compact selfadjoint $K$, let $\lambda_j^+(K)$ list its
positive eigenvalues and let $\lambda_j^-(K)$ list the absolute
values of its negative eigenvalues, both in nonincreasing order.
Thus $\lambda_j^\pm(K)=\mu(j,K_\pm)$.
Zanin and Sukochev \cite[Theorem~1.2]{SZ23} answered this question by proving, for
real-valued $f\in L\log L(X)$, that
\[
 \lim_{j\to\infty}(j+1)
 \lambda_j^\pm\bigl((1-\Delta_g)^{-\frac{d}{4}}M_f(1-\Delta_g)^{-\frac{d}{4}}\bigr)
 =\frac{|\mathbb S^{d-1}|}{d(2\pi)^d}
       \int_X f_\pm\,d\operatorname{vol}_g.
\]
These limits give a constructive integration formula
and imply \eqref{eq:intro-classical-integration} for $L\log L$
potentials. The symmetric sandwich allows potentials in
$L\log L(X)\setminus L_2(X)$.
For such $f$, the one-sided product $M_f(1-\Delta_g)^{-d/2}$
cannot be represented by a bounded operator on $L_2(X)$. Rozenblum \cite[Theorem~6.6]{R22} proved eigenvalue
asymptotics for scalar Birman--Schwinger operators with
$L\log L$ densities on Ahlfors-regular rectifiable sets.
Ponge \cite[Theorems~4.1 and~6.1]{P22} proved Weyl's laws
and integration formulas for matrix-valued $L\log L$ potentials
between pseudodifferential operators of order $-d/2$
on closed manifolds.

On quantum tori (also called ``noncommutative tori"), 
McDonald, Sukochev and Zanin \cite[Theorem~6.5]{MSZ19}
proved a trace formula for $A(1-\Delta)^{-d/2}$, where
$A\in\Pi$. Here $\Pi$ is a $C^*$-algebra containing all
classical pseudodifferential operators of order zero
on $\mathbb T_\theta^d$.
The definition of $\Pi$ and the principal symbol map are given in
Section~\ref{subsec:principal-symbol-algebra}. Their
work on spectral triples \cite{MSZ22} also yields spectral asymptotics
and semiclassical Weyl's laws. Sukochev, Xiong and Zanin \cite[Theorems~1.1--1.2]{SXZ23}
proved singular-value asymptotics for $T(-\Delta)^{-1/2}$,
where $T\in\Pi$.
They applied these formulas to quantised derivatives. Zhang \cite[Corollary~5.2 and Theorem~5.8]{Z26}
proved singular-value asymptotics for $A(1-\Delta)^{-m/2}$
for every $m>0$ and $A\in\Pi$, together with positive
and negative eigenvalue asymptotics for
$(1-\Delta)^{-m/4}A(1-\Delta)^{-m/4}$ when $A=A^*$. These developments motivate the
$L\log L$ endpoint problem on quantum tori: an endpoint Cwikel
estimate provides the continuity needed to pass from bounded
potentials to unbounded Orlicz potentials in spectral asymptotics.

\subsection{Main results}

We prove an $L\log L$ endpoint Cwikel estimate on quantum tori
with a constant only depending on dimension.
We apply the estimate to obtain a critical exponential
Sobolev embedding and to prove Weyl's laws and
Connes' integration formulas for $L\log L$ potentials.

We begin with an abstract weak Hilbert--Schmidt estimate.
Let $(\mathcal M,\tau)$ be a finite von Neumann algebra with a normal
faithful tracial state, and put $H=L_2(\mathcal M,\tau)$.
We write $\lambda$ for left multiplication on $H$ and take all
noncommutative $L_p$ and Orlicz norms with respect to $\tau$. 
\par The main assumption on $\M$ is that $H=L_2(\M)$ admits an at most countable orthonormal basis
$\{u_j\}$ consisting of unitaries in $\mathcal M$.
Let $G$ be the diagonal operator given by $Gu_j=g_ju_j$, where
$g=\{g_j\}\in\ell_{2,\infty}$.
Here $\ell_{2,\infty}$ is the usual weak $\ell_2$ space. 

\begin{theorem}[cf.\ Theorem~\ref{thm:abstract-endpoint}]
\label{thm:intro-abstract-endpoint}
Suppose the von Neumann algebra $\M$ and the diagonal operator $G$ satisfy the assumptions above.
Let $f\geq0$ and 
$f^2\in L\log L(\mathcal M)$. 
Then $\lambda(f)G$
has a unique bounded extension in $\mathcal{L}_{2,\infty}(L_2(\M))$, and
\begin{equation*}
 \|\lambda(f)G\|_{2,\infty}
 \lesssim \|g\|_{\ell_{2,\infty}}\|f\|_{L\log L^{(2)}(\mathcal M)},
\end{equation*}
where $L\log L^{(2)}(\M)$ is the $2$-convexification of $L\log L(\M)$,
with $\|f\|_{L\log L^{(2)}}=\|f^2\|_{L\log L}^{1/2}$. 
\end{theorem}

The proof decomposes the product into a finite-rank term, a bounded
cutoff, and a Hilbert--Schmidt remainder.
The estimate for the bounded cutoff is inspired by the multiplicative
convolution argument of Hundertmark, Kunstmann, Ried and Vugalter \cite{HKRV23}.
The estimate for the Hilbert--Schmidt norm of the remainder is via the orthonormal unitary basis $\{u_j\}$ of $L_2(\M)$,
while a truncated harmonic sum produces the logarithmic weight. 
Unlike the classical proof, the argument requires neither a spatial covering nor a prior critical Sobolev embedding.

\medskip

We now apply this estimate to quantum tori. For
$d\geq2$ and a real antisymmetric $d\times d$ matrix $\theta$, let
$\mathbb T_\theta^d$ be the quantum torus with canonical tracial
state $\tau$, normalized by $\tau(1)=1$. Its generating unitaries
$U_1,\ldots,U_d$ satisfy $U_jU_k=e^{\mathrm i\theta_{jk}}U_kU_j$
and give the orthonormal Fourier basis
$\{U^n=U_1^{n_1}\cdots U_d^{n_d}: n\in\mathbb Z^d\}$ of
$L_2(\mathbb T_\theta^d,\tau)$. 

We write  $\Delta U^n=-|n|^2U^n$. The diagonal symbol $g_n=(1+|n|^2)^{-d/4}$ of
$G=(1-\Delta)^{-d/4}$ belongs to $\ell_{2,\infty}(\mathbb Z^d)$,
with a bound depending only on $d$.
Applying Theorem~\ref{thm:intro-abstract-endpoint} in this setting
and using factorization gives the following endpoint estimate.

\begin{theorem}[cf.\ Theorem~\ref{thm:torus-endpoint}]
\label{thm:intro-quantum-endpoint}
Let $d\geq2$ and let $\theta$ be a real antisymmetric $d\times d$
matrix. For every $x\in L\log L(\mathbb T_\theta^d)$, the operator
$(1-\Delta)^{-d/4}\lambda(x)(1-\Delta)^{-d/4}$ belongs to
$\mathcal{L}_{1,\infty}(L_2(\TT_\theta^d))$ and satisfies
\begin{equation}
\label{eq:intro-quantum-endpoint}
 \big\|(1-\Delta)^{-\frac{d}{4}}\lambda(x)(1-\Delta)^{-\frac{d}{4}}\big\|_{1,\infty}
 \leq C_d\|x\|_{L\log L(\mathbb T_\theta^d)}.
\end{equation}
The constant $C_d$ depends only on $d$ and is independent of $\theta$.
\end{theorem}

This extends \eqref{eq:intro-classical-endpoint} to quantum tori
and answers the question raised in \cite[Remark~6.8]{MP22}.
Theorem~\ref{thm:intro-quantum-endpoint} replaces the $L_p$ 
hypothesis in \cite[Theorem~6.5(2)]{MP22} with $L\log L$, including potentials in no $L_p(p>1)$; 
similar to \eqref{eq:intro-quantum-endpoint}, the corresponding bound for Riesz-potential sandwich $(-\Delta)^{-d/4}\lambda(x)(-\Delta)^{-d/4}$
also follows from Theorem \ref{thm:abstract-endpoint}.
Moreover, we give the same estimates for pseudodifferential operators of order $-d/2$ on both sides,
extending the critical case of \cite[Proposition~4.11(2)]{MP23}, see Corollary~\ref{cor:pdo-sandwich-endpoint}.
When $\theta=0$, our argument also gives an elementary proof of the classical estimate \eqref{eq:intro-classical-endpoint}.

\medskip

Theorem \ref{thm:intro-abstract-endpoint} also recovers the Cwikel estimates for
matrix-valued potentials in \cite[Proposition~3.13 and Corollary~3.15]{P22}. 
Indeed, one can first apply it to the case 
\begin{equation*}
      \M=L_\infty(\mathbb T^d)\mathbin{\overline\otimes}\mathbb M_r, \quad \text{ with normalized trace } \tau=\int_{\TT^d}\otimes\ \frac 1r \Tr, 
\end{equation*}
and let $G=(1-\Delta)^{-d/4}\otimes \operatorname{id}_{\mathbb{M}_r}$. 
The tensor product of the Fourier basis and a unitary orthonormal basis of $\mathbb M_r$ forms a unitary orthonormal basis of $L_2(\M,\tau)$ on which $G$ is diagonal. 
The local trivializations and partition of unity used in \cite[Lemma~3.9]{P22} then yield \cite[Corollary 3.15]{P22}, i.e. 
\begin{equation*}
      \|\Delta_E^{-d/4}u\Delta_E^{-d/4}\|_{1,\infty} \leq C_{\nabla}\Norm{u}_{L\log L(X;\operatorname{End}E)} 
\end{equation*}
for a Hermitian vector bundle $E$ over a closed $d$-dimensional manifold $X$ with connection $\nabla$, where $\Delta_E=\nabla^*\nabla$. 
A usual factorization argument then implies the same estimate for pseudodifferential operator sandwiches in \cite[Proposition~3.13]{P22}. 

\medskip

The endpoint estimate in Theorem~\ref{thm:intro-quantum-endpoint} in turn
yields a critical Sobolev inequality on quantum tori.
Let $W_2^{d/2}(\mathbb T_\theta^d)=\operatorname{Dom}((1-\Delta)^{d/4})$
with norm $\|u\|_{W_2^{d/2}}=\|(1-\Delta)^{d/4}u\|_2$, and let
$\exp(L_2)$ be the Orlicz space associated with $e^{t^2}-1$.
The operator norm bound that follows from \eqref{eq:intro-quantum-endpoint}, combined with
K\"othe duality, yields the following critical embedding.
\begin{corollary}[cf.\ Corollary~\ref{cor:critical sobolev}]
\label{cor:intro-critical-sobolev}
For $d\geq2$, we have
\begin{equation*}
 W_2^{\frac{d}{2}}(\mathbb T_\theta^d)
 \hookrightarrow\exp(L_2)(\mathbb T_\theta^d),\qquad
 \|u\|_{\exp(L_2)}\leq C_d'\|u\|_{W_2^{d/2}},
\end{equation*}
for every $u\in W_2^{d/2}(\mathbb T_\theta^d)$.
\end{corollary}
This extends the classical critical exponential Sobolev embedding theorem in \cite{Str1972} to quantum tori; see also earlier estimates in \cite{Yudovich1961}.
Moreover, for each positive $x\in L\log L(\mathbb T_\theta^d)$ and integer
$n\geq1$, Corollary~\ref{cor:sobolev-subspace} provides a closed
subspace $F_n\subset W_2^{d/2}(\mathbb T_\theta^d)$ of codimension
at most $n$ on which
\[
 \tau(u^*xu)\leq\frac{C_d}{n}\|x\|_{L\log L}\|u\|_{W_2^{d/2}}^2,
 \qquad u\in F_n.
\]
This is a noncommutative counterpart of the weighted approximation estimate \eqref{eq:Solomyak approximation}.

\medskip

For the spectral applications, assume $d\geq2$. Combining the endpoint
estimate with Proposition \ref{prop:ow-bounded-input}, we obtain Weyl's laws for
$S(1-\Delta)^{-d/4}\lambda(x)(1-\Delta)^{-d/4}T$,
where $x\in L\log L(\mathbb T_\theta^d)$ and $S,T\in\Pi$. Put \(H=L_2(\mathbb T_\theta^d,\tau)\).

Write $C(\mathbb T_\theta^d)$ for the $C^*$-algebra generated by
$U_1,\ldots,U_d$. For $g\in C(\mathbb S^{d-1})$, let $\pi(g)$ be
the Fourier multiplier given by
$\pi(g)U^n=g(n/|n|)U^n$ for $n\neq0$ and
$\pi(g)1=0$.
Then $\Pi$ is generated by $\lambda(C(\mathbb T_\theta^d))$ and
$\pi(C(\mathbb S^{d-1}))$. Its principal symbol map is the
$*$-homomorphism
\[
 \begin{aligned}
 \operatorname{sym}:\Pi&\longrightarrow
C(\mathbb T_\theta^d)\otimes_{\min}C(\mathbb S^{d-1}),\\
 \operatorname{sym}(\lambda(y)\pi(g))(s)&=y\,g(s),
 \qquad y\in C(\mathbb T_\theta^d),\quad g\in C(\mathbb S^{d-1});
 \end{aligned}
\]
see Section~\ref{subsec:principal-symbol-algebra}.
We write $(\mathcal{L}_{1,\infty}(H))_0$ for the closure of the
finite-rank operators in the weak trace class quasi-norm.

\begin{theorem}[cf.\ Theorem~\ref{thm:ow-sandwich-weyl}]
\label{thm:intro-sandwich-weyl}
Let $d\geq2$, $S,T\in\Pi$, $x\in L\log L(\mathbb T_\theta^d)$,
and let $R\in(\mathcal{L}_{1,\infty}(H))_0$. Set
\[
 \begin{aligned}
 K&=S(1-\Delta)^{-\frac{d}{4}}\lambda(x)(1-\Delta)^{-\frac{d}{4}}T+R,\\
 a(s)&=\operatorname{sym}(S)(s)x\operatorname{sym}(T)(s).
 \end{aligned}
\]
Then $K\in\mathcal{L}_{1,\infty}(H)$,
$a\in C(\mathbb S^{d-1};L_1(\mathbb T_\theta^d))$, and
\begin{equation}
\label{eq:intro-weyl}
 \lim_{j\to\infty}(j+1)\mu(j,K)
 =\frac{1}{d}\int_{\mathbb S^{d-1}}\tau(|a(s)|)\,ds.
\end{equation}
The function $a$ is determined, up to equality almost everywhere,
by $K$ modulo $(\mathcal L_{1,\infty}(H))_0$.
We call $a$ the generalized principal symbol of order $-d$
of $K$ on $\mathbb S^{d-1}$.
If $K=K^*$, then $a=a^*$ almost everywhere and
\[
 \lim_{j\to\infty}(j+1)\lambda_j^\pm(K)
 =\frac{1}{d}\int_{\mathbb S^{d-1}}\tau(a(s)_\pm)\,ds.
\]
\end{theorem}

Theorem~\ref{thm:intro-sandwich-weyl} also gives Weyl's laws
for $Q\lambda(x)P$.
Here $P$ and $Q$ are classical pseudodifferential operators
of order $-d/2$ on $\mathbb T_\theta^d$, and
$x\in L\log L(\mathbb T_\theta^d)$.
If $p$ and $q$ are their principal symbols restricted to
$\mathbb S^{d-1}$, then $a(s)=q(s)xp(s)$;
see Corollary~\ref{cor:ow-classical-factors}.
For these sandwiches,
\cite[Theorem~5.2(2)]{MP23} gives strong measurability and
the trace formula for $L_p$ potentials, $p>1$.
Theorem~\ref{thm:ow-connes} extends the trace formula to
$L\log L$ potentials.
Corollary~\ref{cor:ow-classical-factors} gives singular-value
asymptotics for $Q\lambda(x)P$ and, when $x=x^*$,
positive and negative eigenvalue asymptotics for
$P^*\lambda(x)P$.
The spectral and trace formulas give quantum torus counterparts of
\cite[Theorems~4.1 and~6.1]{P22}.

Consequently, Theorem~\ref{thm:ow-connes} shows that every continuous
normalized trace $\varphi$ on $\mathcal L_{1,\infty}(H)$ satisfies
\[
 \varphi(K)=\frac{1}{d}\int_{\mathbb S^{d-1}}\tau(a(s))\,ds,
 \qquad
 \varphi(|K|)=\frac{1}{d}\int_{\mathbb S^{d-1}}\tau(|a(s)|)\,ds.
\]
For $K=A(1-\Delta)^{-d/2}$, $A\in\Pi$, the
trace formula recovers \cite[Theorem~6.5]{MSZ19} for every
continuous normalized trace, while \eqref{eq:intro-weyl} gives
ordinary singular-value asymptotics.

For the basic sandwich, these results yield
\[
 \begin{aligned}
 \lim_{j\to\infty}(j+1)
 \lambda_j^\pm\bigl((1-\Delta)^{-\frac{d}{4}}\lambda(x)(1-\Delta)^{-\frac{d}{4}}\bigr)
 &=\frac{|\mathbb S^{d-1}|}{d}\tau(x_\pm),\qquad x=x^*,\\
 \varphi\bigl((1-\Delta)^{-\frac{d}{4}}\lambda(x)(1-\Delta)^{-\frac{d}{4}}\bigr)
 &=\frac{|\mathbb S^{d-1}|}{d}\tau(x),
 \end{aligned}
\]
for $x\in L\log L(\mathbb T_\theta^d)$ and every continuous
normalized trace $\varphi$ on $\mathcal L_{1,\infty}(H)$.
For the basic Riesz-potential sandwich, the spectral asymptotics
and integration formulas in
\cite[Proposition~8.5 and Theorem~8.8]{P26NCG}
assume $x\in L_p(\mathbb T_\theta^d)$ for some $p>1$.
Theorems~\ref{thm:ow-sandwich-weyl}
and~\ref{thm:ow-connes} extend these formulas to
$x\in L\log L(\mathbb T_\theta^d)$.
They also allow outer factors $S,T\in\Pi$.
The same results hold for Bessel potentials.

Finally, for selfadjoint $x\in L\log L(\mathbb T_\theta^d)$ and
$h>0$, the form sum $h^d(1-\Delta)^{d/2}+\lambda(x)$ defines a
semibounded selfadjoint operator with compact resolvent. Write
$N_-$ for its number of strictly negative eigenvalues, counted with
multiplicities. Combining the Birman--Schwinger principle with the eigenvalue asymptotics in Theorem \ref{thm:intro-sandwich-weyl} gives the semiclassical law
\[
 \lim_{h\downarrow0}h^dN_-\bigl(h^d(1-\Delta)^{\frac{d}{2}}+\lambda(x)\bigr)
 =\frac{|\mathbb S^{d-1}|}{d}\tau(x_-),
 \qquad x=x^*\in L\log L(\mathbb T_\theta^d),
\]
as well as an $L\log L$ bound for the negative eigenvalue count
(Corollary~\ref{cor:ow-semiclassical}).

\medskip 

\begingroup
\renewcommand{\addcontentsline}[3]{}
\subsection*{Organization of the paper} 
Our paper is organized as follows. 
Section~\ref{sec:preliminaries} fixes the notation and recalls the
analytic background. Section~\ref{sec:cwikel-estimates} proves the
abstract endpoint estimate, and Section~\ref{sec:consequences}
develops its quantum torus and Sobolev consequences.
Section~\ref{sec:orlicz-weyl} contains Weyl's laws, integration
formulas, and a semiclassical application.
\par
\endgroup

\section{Preliminaries}
\label{sec:preliminaries}

\subsection{Noncommutative spaces}

Let $\mathcal M\subset\mathcal B(\mathcal H)$ be a von Neumann algebra
with a normal, semifinite, faithful trace $\tau$. We use the measurable-operator realization
of noncommutative $L_p$-spaces; see \cite{PX03,Xu07,LSZ12,DPS23}.
Write $L_0(\mathcal M)$ for the algebra of closed densely defined
$\tau$-measurable operators affiliated with $\mathcal M$. Here affiliation means commutation with
the unitaries of $\mathcal M'$, including invariance of the domain;
$\tau$-measurability means that, for every $\varepsilon>0$, the domain
contains $e\mathcal H$ for a projection $e\in\mathcal M$ with
$\tau(1-e)<\varepsilon$.

For $x\in L_0(\mathcal M)$, set
\[
 d_x(s)=\tau\bigl(\mathbf1_{(s,\infty)}(|x|)\bigr),\qquad
 \mu(t,x)=\inf\{s\geq0:d_x(s)\leq t\},\qquad s,t>0,
\]
where $d_x(0)=\tau(\mathbf1_{(0,\infty)}(|x|))$, and put
$\mu(0,x)=\lim_{t\downarrow0}\mu(t,x)$.
For $1\leq p<\infty$,
\[
 L_p(\mathcal M)
 =\{x\in L_0(\mathcal M):\tau(|x|^p)<\infty\},\qquad
 \|x\|_p
 =\left(\int_0^\infty\mu(t,x)^p\,dt\right)^{\frac{1}{p}}.
\]
We set $L_\infty(\mathcal M)=\mathcal M$, with its operator norm.
Products of measurable operators are understood as their closures.
The noncommutative form of H\"older's inequality gives
\[
 \|xy\|_r\leq\|x\|_p\|y\|_q,
 \qquad \frac{1}{r}=\frac{1}{p}+\frac{1}{q},
 \qquad 1\leq p,q,r\leq\infty.
\]
We also use
\[
 \mu(t,x)=\mu(t,x^*)=\mu(t,|x|),\qquad
 \mu(t,|x|^\alpha)=\mu(t,x)^\alpha\quad(\alpha>0).
\]

For $0<p<\infty$, the weak $L_p$-space has quasi-norm
\[
 \|x\|_{L_{p,\infty}(\mathcal M)}
 =\sup_{t>0}t^{\frac{1}{p}}\mu(t,x)
 =\sup_{s>0}s\,d_x(s)^{\frac{1}{p}}.
\]
When $\tau(1)=1$, we use the Orlicz space associated with
\[
 \Phi(t)=t\log(e+t),\qquad t\geq0.
\]
It is denoted by $L_\Phi(\mathcal M)=L\log L(\mathcal M)$ and consists
of all $x\in L_0(\mathcal M)$ for which
$\tau\big(\Phi\big(\frac{|x|}{c}\big)\big)<\infty$ for some $c>0$. Its Luxemburg norm is
\begin{equation}
\label{eq:prelim-luxemburg}
 \|x\|_{L_\Phi(\mathcal M)}
 =\inf\Bigl\{c>0:\tau\Bigl(\Phi\Bigl(\frac{|x|}{c}\Bigr)\Bigr)\leq1\Bigr\}.
\end{equation}
This is a Banach norm; see \cite[Section~2]{LM11}. We shall use
\begin{equation*}
 \|x\|_1\leq\|x\|_{L_\Phi},\qquad
 \|x^*\|_{L_\Phi}=\|x\|_{L_\Phi},\qquad
 \|axb\|_{L_\Phi}\leq\|a\|_\infty\|x\|_{L_\Phi}\|b\|_\infty
\end{equation*}
for $a,b\in\mathcal M$.
The first inequality follows from $\Phi(t)\geq t$.
The other two follow from $\mu(t,x^*)=\mu(t,x)$ and
\[
 \mu(t,axb)\leq\|a\|_\infty\mu(t,x)\|b\|_\infty,
 \qquad t>0.
\]

For each $1<p<\infty$, the continuous inclusions
\[
 L_\infty(\mathcal M)\subset L_p(\mathcal M)
 \subset L_\Phi(\mathcal M)\subset L_1(\mathcal M)
\]
follow from $\tau(1)=1$ and $\Phi(t)\leq C_p(t+t^p)$.

\subsection{Weak Schatten ideals and Weyl's laws}
\label{subsec:weak-schatten}

Let $H$ be a separable Hilbert space. Operator ideals on $H$ are
defined using the ordinary trace $\operatorname{Tr}_H$ on $\mathcal B(H)$.
For a compact operator $T$, its singular values are $\mu(j,T)$,
$j\in\mathbb N_0$, in nonincreasing order. Thus
\[
 \mathcal{L}_{p,\infty}(H)
 =L_{p,\infty}(\mathcal B(H),\operatorname{Tr}_H),\qquad
 \|T\|_{p,\infty}
 =\sup_{j\geq0}(j+1)^{\frac{1}{p}}\mu(j,T).
\]
We write $\mathcal{L}_p(H)$ for the usual Schatten ideal. In particular,
$\mathcal{L}_{1,\infty}(H)$ is the weak trace class.
The closed subspace
\[
 (\mathcal{L}_{p,\infty}(H))_0
 =\{T\in\mathcal{L}_{p,\infty}(H):
       (j+1)^{\frac{1}{p}}\mu(j,T)\longrightarrow0\}
\]
is the closure of the finite-rank operators in this quasi-norm.

For compact operators $A,B$, the singular-value inequalities give
\begin{equation}
\label{eq:prelim-singular-inequalities}
 \mu(j+k,A+B)\leq\mu(j,A)+\mu(k,B),\qquad
 \mu(j+k,AB)\leq\mu(j,A)\mu(k,B).
\end{equation}
We also have the following form of H\"older's inequality:
\begin{equation}\label{eq:prelim-weak-holde}
 \|TS\|_{r,\infty}
 \leq
 \frac{(p+q)^{\frac1p+\frac1q}}
             {q^{\frac1p}p^{\frac1q}}
 \|T\|_{p,\infty}\|S\|_{q,\infty},
\end{equation}
where $p,q,r>0$ and
$\frac1r=\frac1p+\frac1q$.
This holds for $T\in\mathcal L_{p,\infty}$ and
$S\in\mathcal L_{q,\infty}$.
The constant is optimal; see
\cite[Theorem~4]{SZHolder2021}.
The estimate \eqref{eq:prelim-weak-holde}, combined with
finite-rank approximation, gives
\begin{equation}
\label{eq:prelim-small-product}
 (\mathcal{L}_{p,\infty})_0\mathcal{L}_{q,\infty}
 +\mathcal{L}_{p,\infty}(\mathcal{L}_{q,\infty})_0
 \subset(\mathcal{L}_{r,\infty})_0.
\end{equation}

For $T\in\mathcal L_{p,\infty}(H)$ and $A,B\in\mathcal B(H)$,
\begin{equation}
\label{eq:prelim-ideal-property}
 \|ATB\|_{p,\infty}\leq\|A\|\|T\|_{p,\infty}\|B\|.
\end{equation}
Bounded left and right multiplication also preserves
$(\mathcal L_{p,\infty}(H))_0$.

For compact selfadjoint $K$, let $\lambda_j^\pm(K)$ be its positive
eigenvalues and the absolute values of its negative eigenvalues,
respectively, in nonincreasing order, with multiplicities and zeros
appended to finite lists. Thus $\lambda_j^\pm(K)=\mu(j,K_\pm)$.
Glazman's lemma (see, e.g., \cite[Lemma~7.1]{MP22}) gives
\begin{equation}
\label{eq:prelim-signed-inequality}
 \lambda_{j+k}^\pm(B+C)
 \leq\lambda_j^\pm(B)+\mu(k,C)
\end{equation}
for compact selfadjoint $B,C$.

We recall the spectral approximation results used below.
The first lemma follows from \cite[Proposition~2.2]{FK86}.

\begin{lemma}
\label{lem:distribution-singular-asymptotics}
Let $p>0$ and $0\leq A\in\mathcal{L}_{p,\infty}(H)$. Either one of
the following limits exists if and only if the other does, and
\[
 \lim_{s\downarrow0}s^p\operatorname{Tr}_H
       \bigl(\mathbf1_{(s,\infty)}(A)\bigr)
 =\lim_{j\to\infty}(j+1)\mu(j,A)^p.
\]
\end{lemma}

The next two lemmas are due to Birman and Solomyak;
see \cite[Section~11.6]{BS1987}.

\begin{lemma}
\label{lem:weak-weyl-stability}
Let $p>0$, $A\in\mathcal{L}_{p,\infty}(H)$, and
$R\in(\mathcal{L}_{p,\infty}(H))_0$. If
$\lim_{j\to\infty}(j+1)^{\frac{1}{p}}\mu(j,A)$ exists, then
\[
 \lim_{j\to\infty}(j+1)^{\frac{1}{p}}\mu(j,A+R)
 =\lim_{j\to\infty}(j+1)^{\frac{1}{p}}\mu(j,A).
\]
\end{lemma}

\begin{lemma}
\label{lem:weak-ideal-approximation}
Let $p>0$ and suppose that $A_n,A\in\mathcal{L}_{p,\infty}(H)$ satisfy
$\|A_n-A\|_{p,\infty}\to0$. If each limit
\[
 c_n=\lim_{j\to\infty}(j+1)^{\frac{1}{p}}\mu(j,A_n)
\]
exists, then $c_n$ converges and
\[
 \lim_{j\to\infty}(j+1)^{\frac{1}{p}}\mu(j,A)=\lim_{n\to\infty}c_n.
\]
\end{lemma}

We also use the following result for selfadjoint operators.
It follows from \cite[Corollary~4.3 and Lemmas~4.4--4.5]{SZ23}.

\begin{lemma}
\label{lem:ow-signed-stability}
Suppose that $A_n=A_n^*,A=A^*\in\mathcal{L}_{1,\infty}(H)$,
$\|A_n-A\|_{1,\infty}\to0$, and
\[
 c_n^\pm=\lim_{j\to\infty}(j+1)\lambda_j^\pm(A_n)
 \longrightarrow c^\pm.
\]
Then $(j+1)\lambda_j^\pm(A)\to c^\pm$.
If $R=R^*\in(\mathcal{L}_{1,\infty}(H))_0$, then
\[
\lim_{j\to\infty}(j+1)\lambda_j^\pm(A+R)
=\lim_{j\to\infty}(j+1)\lambda_j^\pm(A),
\]
provided that the limit on the right exists.
\end{lemma}

\subsection{Quantum tori $\mathbb{T}_\theta^d$}
\label{subsec:quantum-tori}

Let $d\geq2$ and let $\theta$ be a real antisymmetric $d\times d$
matrix. The quantum torus $C(\mathbb T_\theta^d)$ is the universal
unital $C^*$-algebra generated by unitaries $U_1,\ldots,U_d$ satisfying
\[
 U_jU_k=e^{\mathrm i\theta_{jk}}U_kU_j.
\]
See \cite[Section~2.2]{CXY13}. Write $U^n=U_1^{n_1}\cdots U_d^{n_d}$ for $n\in\mathbb Z^d$.
The canonical faithful tracial state is determined by
$\tau(U^n)=\delta_{n,0}$. We represent $C(\mathbb T_\theta^d)$ on
$H=L_2(\mathbb T_\theta^d,\tau)$ by left multiplication, denoted by
$\lambda$. The weak operator closure of $\lambda(C(\mathbb T_\theta^d))$
is $L_\infty(\mathbb T_\theta^d)$, and the trace $\tau$ extends to a normal faithful tracial state on this algebra.
The spaces $L_p(\mathbb T_\theta^d)$ and $L\log L(\mathbb T_\theta^d)$
are defined with respect to this state. We use the inner product
$\langle\xi,\eta\rangle_H=\tau(\eta^*\xi)$.

For $x\in L_1(\mathbb T_\theta^d)$, set
\[
 \widehat x(n)=\tau\bigl(x(U^n)^*\bigr).
\]
The monomials $U^n$ form an orthonormal basis of $H$.
Thus $x=\sum_{n\in\mathbb Z^d}\widehat x(n)U^n$ in $H$ for $x\in H$.
Parseval's identity gives
\begin{equation}
\label{eq:prelim-parseval}
 \|x\|_2^2=\sum_{n\in\mathbb Z^d}|\widehat x(n)|^2,
 \qquad x\in H.
\end{equation}

The canonical torus action, written using its $2\pi$-periodic lift, is
\[
 \alpha_t(U^n)=e^{\mathrm i t\cdot n}U^n,
 \qquad t\in\mathbb R^d.
\]
It extends to a trace-preserving action on
$L_\infty(\mathbb T_\theta^d)$ and to isometries on each
$L_p(\mathbb T_\theta^d)$; see
\cite[Section~2.3 and Proposition~2.1]{CXY13}.

We denote by $\mathcal S(\mathbb T_\theta^d)$ the algebra of
elements with rapidly decreasing Fourier coefficients.
Equivalently, an element $x\in C(\mathbb T_\theta^d)$ belongs to
$\mathcal S(\mathbb T_\theta^d)$ if and only if
$t\mapsto\alpha_t(x)$ is infinitely differentiable
in the operator norm; see \cite[Section~2.1]{XXY18}.
We denote by $\mathcal D'(\mathbb T_\theta^d)$ the space of
distributional Fourier series, whose coefficients have
at most polynomial growth.

The selfadjoint Fourier operators $D_j$, taken on their maximal domains
in $H$, and the Laplacian are given by
\[
 D_jU^n=n_jU^n,\qquad
 \partial_j=\mathrm iD_j,\qquad
 \Delta=\sum_{j=1}^d\partial_j^2=-\sum_{j=1}^dD_j^2.
\]
On smooth elements, $D_j$ satisfies the Leibniz rule. We use
$D^\beta=D_1^{\beta_1}\cdots D_d^{\beta_d}$ and
$|\beta|=\beta_1+\cdots+\beta_d$. The seminorms
\begin{equation}
\label{eq:prelim-smooth-seminorms}
 p_N(x)=\max_{|\beta|\leq N}\|D^\beta x\|_\infty,
 \qquad N\in\mathbb N_0,
\end{equation}
define the smooth topology on $\mathcal S(\mathbb T_\theta^d)$.
For $u\in\mathcal S(\mathbb T_\theta^d)$, put
$u_M=\sum_{|n|\leq M}\widehat u(n)U^n$.
Rapid decay of the Fourier coefficients gives
\begin{equation}
\label{eq:prelim-smooth-fourier-density}
 p_N(u-u_M)
 \leq\sum_{|n|>M}(1+|n|)^N|\widehat u(n)|
 \longrightarrow0
 \qquad(N\in\mathbb N_0).
\end{equation}
Thus Fourier polynomials are dense in this topology.
For $1<p<\infty$,
\[
 \mathcal{S}(\mathbb T_\theta^d)\subset C(\mathbb T_\theta^d)
 \subset L_\infty(\mathbb T_\theta^d)\subset L_p(\mathbb T_\theta^d)\subset L\log L(\mathbb T_\theta^d)
 \subset L_1(\mathbb T_\theta^d),
\]
with continuous inclusions.

For $\alpha\in\mathbb R$, as in \cite[Section~2.2]{XXY18}, we define the powers of $1-\Delta$ by
\[
 (1-\Delta)^{\frac{\alpha}{2}}U^n=(1+|n|^2)^{\frac{\alpha}{2}}U^n,
\]
and set
\begin{equation}
\label{eq:prelim-sobolev}
 \begin{gathered}
 W_2^\alpha(\mathbb T_\theta^d)
 =\{u\in\mathcal D'(\mathbb T_\theta^d):(1-\Delta)^{\frac{\alpha}{2}}u\in H\},\\
 \|u\|_{W_2^\alpha}
 =\bigg(\sum_n(1+|n|^2)^\alpha|\widehat u(n)|^2\bigg)^{\frac 12}.
 \end{gathered}
\end{equation}
By \eqref{eq:prelim-sobolev},
$(1-\Delta)^{-\alpha/2}:H\to W_2^\alpha$ is an isometric isomorphism.
Fourier polynomials are dense in $W_2^\alpha$.
For $\alpha>\beta$, the embedding $W_2^\alpha\hookrightarrow W_2^\beta$
is compact, since the diagonal weights
$(1+|n|^2)^{\frac{\beta-\alpha}{2}}$ tend to zero.

\subsection{A $C^*$-algebra of pseudodifferential operators.}
\label{subsec:principal-symbol-algebra}
In this subsection, let $d\geq2$.
We use the $C^*$-algebraic principal symbol construction of
\cite{SZ18,MSZ19}, as in \cite{SXZ23}.
For $g\in C(\mathbb S^{d-1})$, define
\[
 \pi(g)U^n=g\Bigl(\frac{n}{|n|}\Bigr)U^n,
 \quad n\neq0,\qquad \pi(g)1=0.
\]
This is a faithful $*$-representation. It satisfies
$\pi(1)=I-P_0$, where $P_0$ is the orthogonal projection onto
the constant vectors. Let
\[
 \Pi=C^*\bigl(\lambda(C(\mathbb T_\theta^d)),
               \pi(C(\mathbb S^{d-1}))\bigr)\subset\mathcal B(H).
\]
By \cite[Theorem~3.3 and Lemma~4.5]{MSZ19}, there is a
surjective $*$-homomorphism
\begin{equation*}
 \operatorname{sym}:\Pi\longrightarrow
 C(\mathbb T_\theta^d)\otimes_{\min}C(\mathbb S^{d-1}),\qquad
 \operatorname{sym}(\lambda(y)\pi(g))(s)=y\,g(s),
\end{equation*}
We identify the target algebra with
$C(\mathbb S^{d-1};C(\mathbb T_\theta^d))$.
The kernel is $\Pi\cap\mathcal K(H)$.
The map is contractive and preserves continuous functional calculus.
Since $P_0=I-\pi(1)\in\Pi$, the Fourier matrix units
$\lambda(U^n)P_0\lambda(U^m)^*$ belong to $\Pi$.
Thus $\mathcal K(H)\subset\Pi$.
Changing the value at zero frequency adds an operator of rank at most one,
so it does not change the principal symbol.
For $A=A^*\in\Pi$,
\[
 A_\pm\in\Pi,\qquad
 \operatorname{sym}(A_\pm)(s)=\operatorname{sym}(A)(s)_\pm.
\]

For $m>0$, lattice counting gives
\[
 (1-\Delta)^{-\frac m2}\in\mathcal L_{\frac dm,\infty}(H).
\]
For $A\in\Pi$, \eqref{eq:prelim-ideal-property} implies
\[
 \|A(1-\Delta)^{-\frac m2}\|_{\frac dm,\infty}
 \leq\|A\|\,\|(1-\Delta)^{-\frac m2}\|_{\frac dm,\infty}.
\]
Indeed, for $0<s<1$,
\[
 d_{(1-\Delta)^{-m/2}}(s)
 =\#\{n\in\mathbb Z^d:(1+|n|^2)^{-\frac{m}{2}}>s\}
 \leq C_d s^{-\frac{d}{m}},
\]
and the counting function vanishes for $s\geq1$.

\subsection{Classical pseudodifferential operators on $\mathbb{T}_\theta^d$}
\label{subsec:classical-pseudodifferential}

We use the classical calculus on quantum tori in the normalization
$D_jU^n=n_jU^n$; see \cite{HLP19i,HLP19ii}.
For $m\in\mathbb R$, the symbol class
$S^m(\mathbb R^d;\mathcal{S}(\mathbb T_\theta^d))$ consists of smooth maps
$\sigma$ satisfying
\[
 \|D^\alpha\partial_\xi^\beta\sigma(\xi)\|_\infty
 \leq C_{\alpha,\beta}(1+|\xi|^2)^{\frac{m-|\beta|}{2}}.
\]
See \cite[Definition~3.1]{HLP19i}.
For $\sigma\in S^m(\mathbb R^d;\mathcal{S}(\mathbb T_\theta^d))$, the associated pseudodifferential operator
$P_\sigma$ acts on $\mathcal S(\mathbb T_\theta^d)$ by
\[
 P_\sigma u
 =\sum_{n\in\mathbb Z^d}
   \sigma(n)\widehat u(n)U^n;
\]
see \cite[Proposition~5.9]{HLP19i}.
We denote the class of pseudodifferential operators of order $m$
by $\Psi^m(\mathbb T_\theta^d)$.

\begin{definition}[{\cite[Definitions~3.11, 3.14 and Remark~3.21]{HLP19i}}]
\label{def:classical-pdo-quantum-torus}
A symbol $\sigma\in S^m(\mathbb R^d;\mathcal S(\mathbb T_\theta^d))$ is classical if there are smooth components
$\sigma_{m-j}$ on $\mathbb R^d\setminus\{0\}$, with values in
$\mathcal{S}(\mathbb T_\theta^d)$, such that
\[
 \sigma_{m-j}(t\xi)=t^{m-j}\sigma_{m-j}(\xi),\qquad
 \sigma-\sum_{j=0}^{N-1}\chi\sigma_{m-j}\in S^{m-N}
 \quad(t>0,\ N\geq1).
\]
Here $\chi$ is smooth, vanishes near zero, and equals one outside a
ball. We denote this class of symbols by
$S_{\mathrm{cl}}^m(\mathbb R^d;\mathcal S(\mathbb T_\theta^d))$.
The operators $P_\sigma$ with $\sigma\in S_{\mathrm{cl}}^m(\mathbb R^d;\mathcal S(\mathbb T_\theta^d))$
are called classical pseudodifferential operators of order $m$.
We denote their class by $\Psi_{\mathrm{cl}}^m(\mathbb T_\theta^d)$.
Their principal symbol is $\sigma(P_\sigma)_m=\sigma_m$.
It depends only on $P_\sigma$; see
\cite[Corollary~5.13(ii)]{HLP19i}.
\end{definition}

Order-zero operators are bounded on $H$ by
\cite[Proposition~10.1]{HLP19ii}.
Each $P_\sigma$ acts continuously on $\mathcal S(\mathbb T_\theta^d)$;
see \cite[Proposition~5.4]{HLP19i}.
For $\sigma\in S^m(\mathbb R^d;\mathcal{S}(\mathbb T_\theta^d))$, \cite[Proposition~10.4]{HLP19ii} gives
\begin{equation}\label{eq:prelim-pdo-sobolev-mapping}
 \|P_\sigma u\|_{W_2^s}
 \leq C_{\sigma,s}\|u\|_{W_2^{s+m}},
 \qquad u\in W_2^{s+m}(\mathbb T_\theta^d),\quad s\in\mathbb R.
\end{equation}
The composition and formal adjoint formulas
\cite[Propositions~7.6 and~8.3]{HLP19ii} give
\[
 \begin{aligned}
 \sigma(PQ)_{m+m'}&=\sigma(P)_m\sigma(Q)_{m'},\qquad
 &&P\in\mathrm \Psi_{\mathrm{cl}}^m(\mathbb T_\theta^d),\ Q\in\mathrm \Psi_{\mathrm{cl}}^{m'}(\mathbb T_\theta^d),\\
 \sigma(P^*)_m&=\sigma(P)_m^*.
 \end{aligned}
\]
Moreover, $(1-\Delta)^{m/2}\in\mathrm \Psi_{\mathrm{cl}}^m(\mathbb T_\theta^d)$ has principal symbol
$|\xi|^m$ by \cite[Proposition~5.14]{HLP19i}. For $d\geq2$, \cite[Proposition~5.10]{Z26} gives
\[
 \mathrm \Psi_{\mathrm{cl}}^0(\mathbb T_\theta^d)\subset\Pi,\qquad
 \operatorname{sym}(P)(s)=\sigma(P)_0(s)
 \quad(P\in\mathrm \Psi_{\mathrm{cl}}^0(\mathbb T_\theta^d)).
\]
Consequently, for $d\geq2$, if $m>0$ and
$P\in\mathrm\Psi_{\mathrm{cl}}^{-m}(\mathbb T_\theta^d)$, the operators
$P(1-\Delta)^{m/2}$ and $(1-\Delta)^{m/2}P$, initially defined on smooth elements,
have bounded extensions in $\Pi$, and
\begin{equation}
\label{eq:prelim-classical-factorization}
 \begin{gathered}
 P=\bigl(P(1-\Delta)^{\frac{m}{2}}\bigr)(1-\Delta)^{-\frac{m}{2}}
  =(1-\Delta)^{-\frac{m}{2}}\bigl((1-\Delta)^{\frac{m}{2}}P\bigr),\\
 \operatorname{sym}\bigl(P(1-\Delta)^{\frac{m}{2}}\bigr)(s)
 =\operatorname{sym}\bigl((1-\Delta)^{\frac{m}{2}}P\bigr)(s)=\sigma(P)_{-m}(s).
 \end{gathered}
\end{equation}

\section{An abstract Cwikel estimate at the \texorpdfstring{$L\log L$}{L log L} endpoint}
\label{sec:cwikel-estimates}

We establish an endpoint estimate for the product of a left multiplication operator and a diagonal operator.
The argument applies to finite von Neumann
algebras whose $L_2$ space admits an orthonormal basis of
unitaries. In particular, the Fourier basis
$\{U^n:n\in\mathbb Z^d\}$ of $L_2(\mathbb T_\theta^d)$ has this property.

Let $\mathcal M$ be a finite von Neumann algebra equipped with a normal
faithful tracial state $\tau$, and put $H=L_2(\mathcal M,\tau)$, with
$\langle\xi,\eta\rangle_H=\tau(\eta^*\xi)$.
For $a\in L_0(\mathcal M)$, denote by $\lambda(a)$ the closed operator
of left multiplication on its natural domain,
\[
 \operatorname{Dom}\lambda(a)=\{\xi\in H:a\xi\in H\},
 \qquad \lambda(a)\xi=a\xi.
\]
For bounded $a$, $a\mapsto \lambda(a)$ is the left-regular representation.
We write $\operatorname{Tr}_H$ for the usual operator trace on
$\mathcal B(H)$ and use the notation
\[
 \mathcal{L}_{2,\infty}(H)
   =L_{2,\infty}(\mathcal B(H),\operatorname{Tr}_H),
 \qquad
 \mathcal{L}_2(H)=L_2(\mathcal B(H),\operatorname{Tr}_H).
\]
Thus, for a compact operator $T$ on $H$, the distribution function and
weak Schatten quasi-norm are
\begin{equation}
\label{eq:cwikel-counting-norm}
 d_T(s)=\operatorname{Tr}_H\bigl(\mathbf1_{(s,\infty)}(|T|)\bigr),
 \qquad
 \|T\|_{\mathcal{L}_{2,\infty}(H)}^2
   =\sup_{s>0}s^2d_T(s).
\end{equation}
We use the Luxemburg norm on $L\log L(\mathcal M)$ associated with
$\Phi(t)=t\log(e+t)$, as in \eqref{eq:prelim-luxemburg}.
Write $\log_+t=\max\{\log t,0\}$ for $t>0$ and $\log_+0=0$.

\par Suppose that $H=L_2(\M)$ has an (at most) countable orthonormal basis
$\{u_j\}_{j\geq1}$ consisting of unitaries in $\mathcal M$. 
Let $G$ be the diagonal operator on $H$ given by $Gu_j=g_ju_j$. Write $g=\{g_j\}_{j\geq1}$. 
For a finite or countable sequence $g$, let $\{g_k^*\}_{k\geq1}$ be the decreasing rearrangement of $\{|g_j|\}$, with zeros appended when only finitely many entries are nonzero. We write
\begin{equation*}
  \Norm{g}_{\ell_{2,\infty}}=\sup_{k\geq1}k^{1/2}g_k^*. 
\end{equation*}

\begin{theorem}
\label{thm:abstract-endpoint}
Let $\M$, $H$, $G$ be as above.
If $f\in L_0(\mathcal M)_+$ satisfies $f^2\in L\log L(\mathcal M)$, and $g=\{g_j\}_{j\geq1}\in\ell_{2,\infty}$, then
$\lambda(f)G$, initially defined on the finite linear span of the $u_j$,
has a unique bounded extension to $H$. This extension is compact,
coincides with the natural multiplication product on all of $H$, and satisfies
\begin{equation}
\label{eq:abstract-endpoint}
 \|\lambda(f)G\|_{\mathcal{L}_{2,\infty}(H)}^2
   \le12\|g\|_{\ell_{2,\infty}}^2\|f^2\|_{L\log L(\mathcal M)}.
\end{equation}
\end{theorem}

\begin{proof}
The cases $f=0$ or $g=0$ are immediate. Otherwise, put
$a=\|f^2\|_{L\log L(\mathcal M)}^{1/2}$ and
$b=\|g\|_{\ell_{2,\infty}}$, and write $G=|G|V$, where
$Vu_j=\omega_ju_j$, with $\omega_j=g_j/|g_j|$ if $g_j\ne0$ and
$\omega_j=1$ otherwise. Then $V$ is unitary. Setting
$\tilde f=f/a$ and $\tilde G=|G|/b$, we have, on the natural domains,
\[
 \lambda(f)G=ab\bigl(\lambda(\tilde f)\tilde G\bigr)V.
\]
It suffices to prove the theorem for $\tilde{f}$ and $\tilde{G}$.
The squared weak Schatten quasi-norm then scales by $a^2b^2$.
We may therefore assume
\begin{equation}
\label{eq:cwikel-normalization}
 \|f^2\|_{L\log L(\mathcal M)}=1,\qquad \|g\|_{\ell_{2,\infty}}=1,
 \qquad G\ge0.
\end{equation}

We consider only the basis vectors $u_j$ whose coefficients $g_j$ are nonzero. 
We use $\{u_j\}_{j\in J}$ to denote this enumeration for those $u_j$ with $g_j\neq0$. 
This index set $J$ may be $\{1,\dots, n\}$ or $\{j\in\mathbb{N}:j\geq1\}$.
After relabeling these basis vectors (in nonincreasing order of $g_j$), we may assume
\begin{equation}
  \label{eq:cwikel-g_j<j^(1/2)}
  0<g_j=g_j^*\le j^{-\frac{1}{2}}, \qquad j\in J.
\end{equation}
In the rest of the proof, all indices $j$ range over $J$, with $\{g_j\}_{j\in J}$ satisfying \eqref{eq:cwikel-g_j<j^(1/2)}. 
The original basis vectors with zero diagonal coefficients form an orthonormal basis of $\ker G$. 

By \eqref{eq:cwikel-normalization} and the definition of the Luxemburg norm, 
we have $\tau(\Phi(f^2))\le1$. Since $\Phi(t)\ge t$, 
we have $\tau(f^2)\leq 1$.
Write $\mathcal P$ for the finite linear span of the full original basis. 
Since $f\in L_2(\mathcal M)$ and each $u_j$ is unitary, traciality gives
$\|fu_j\|_H^2=\tau(f^2)\le1$. 
Hence $\lambda(f)G$ is defined on $\mathcal P$. 

\par Fix $s>0$, set $N_s=\lfloor s^{-2}\rfloor$, and let $P_s$ be the
orthogonal projection onto the span of the basis vectors $\{u_j:j\in J, j\leq N_s\}$.
When $N_s=0$, put $P_s=0$. The operator $F_s=\lambda(f)GP_s$ is bounded and has rank at
most $N_s$. We shall decompose the remaining term $\lambda(f)G(I-P_s)$ into a bounded
cutoff of norm at most $s$ and a Hilbert--Schmidt remainder.
\par To construct the bounded cutoff, we adapt the multiplicative
convolution argument in \cite[Theorem~2.1 and Remark~1.4]{HKRV23}.
For $t>0$, define 
\[
 \varphi_s(t)=\sqrt2\,t\mathbf1_{(0,s]}(t),
 \qquad
 \psi(t)=\sqrt2\,t^{-1}\mathbf1_{[1,\infty)}(t),
\]
and set both functions equal to zero at $t=0$. 
On the locally compact abelian group $\RR_+=(0,\infty)$ equipped with the Haar measure $dt/t$, we let 
\begin{equation*}
  m_s(t)=(\varphi_s*\psi)(t). 
\end{equation*}
Direct integration gives
\begin{equation}
\label{eq:cwikel-cutoff-profile}
 m_s(t):=\int_0^\infty\varphi_s\Bigl(\frac{t}{r}\Bigr)\psi(r)\,\frac{dr}{r}
   =\min\Bigl\{t,\frac{s^2}{t}\Bigr\} \qquad (t>0), 
\end{equation}
with $\|\varphi_s\|_{L_2(dr/r)}=s$ and
$\|\psi\|_{L_2(dr/r)}=1$. 
Also set $m_s(0)=0$. 
Let $I$ denote the identity operator on $H$, and put
\[
 X_r=\lambda\Bigl(\varphi_s\Bigl(\frac{f}{r}\Bigr)\Bigr),
 \qquad Y_r=\psi\bigl(rG(I-P_s)\bigr),\qquad r>0.
\]
Then both factors are bounded. Indeed, spectral calculus gives the following strong integrals of positive operators: 
\[
 \int_0^\infty X_rX_r^*\,\frac{dr}{r}
   =s^2\lambda\bigl(\mathbf1_{(0,\infty)}(f)\bigr)\le s^2I,
 \qquad
 \int_0^\infty Y_r^*Y_r\,\frac{dr}{r}\le I.
\]
Consequently, the Cauchy--Schwarz inequality gives
\[
 \int_0^\infty
   |\langle X_rY_r\xi,\eta\rangle_H|\,\frac{dr}{r}
   \le s\|\xi\|_H\|\eta\|_H,
 \qquad \xi,\eta\in H.
\]
The weak operator integral
\begin{equation*}
 Q_s=\int_0^\infty X_rY_r\,\frac{dr}{r}
\end{equation*}
therefore defines a bounded operator with $\|Q_s\|\le s$.
Both $F_s$ and $Q_s$ vanish on $\ker G$.
On each relabeled basis vector, \eqref{eq:cwikel-cutoff-profile} gives 
\begin{equation}
  \label{eq:Q_s def on u_j}
  Q_su_j=m_s(g_jf)u_j \quad (j>N_s); \qquad Q_su_j=0\quad (j\leq N_s).
\end{equation}
Indeed, by \eqref{eq:cwikel-cutoff-profile} and a change of variables, we have 
\begin{equation*}
  m_s(g_jt)=\int_0^\infty\varphi_s\Bigl(\frac{g_j}{r}\Bigr)\psi(rt)\frac{dr}{r}
  =\int_0^\infty\varphi_s\Bigl(\frac{t}{r}\Bigr)\psi(rg_j)\frac{dr}{r}.
\end{equation*}
The scalar integrals over finite intervals increase to
$m_s(g_jt)\le s$, so \eqref{eq:Q_s def on u_j} follows by dominated convergence
in the spectral calculus of $f$.

Set $r_s(t)=t-m_s(t)$ for $t\ge0$. Define the operator $R_s$ by 
\begin{equation}
  \label{eq:R_s def on u_j}
  R_su_j=r_s(g_jf)u_j \quad (j>N_s); \qquad R_su_j=0 \quad (j\leq N_s). 
\end{equation}
Set $R_s=0$ on $\ker G$.
We will show that $R_s$ extends to a Hilbert--Schmidt operator. Since
\[
 0\le r_s(t)^2\le t^2\mathbf1_{(s,\infty)}(t),
\]
the squared remainder $r_s(g_jf)^2$ is bounded above by
$g_j^2f^2\mathbf1_{(s,\infty)}(g_jf)$. Since $N_s=\lfloor s^{-2} \rfloor$ and the $g_j$ satisfy
\eqref{eq:cwikel-g_j<j^(1/2)}, we have the elementary bound
\begin{equation}
\label{eq:cwikel-harmonic-tail}
 \sum_{j>N_s}g_j^2\mathbf1_{\{g_jt>s\}}
   \le\sum_{N_s<j<\frac{t^2}{s^2}}\frac{1}{j}
   \le1+\log_+(2t^2),\qquad t\ge0.
\end{equation}
To see the last inequality, we may assume that the sum is nonempty.
If $s\le1$, then $N_s\ge s^{-2}/2$ and $N_s\ge1$;
integral comparison bounds the sum by
$\log(t^2/(s^2N_s))\le\log_+(2t^2)$.
If $s>1$, then $N_s=0$, and the same comparison after separating the
first term gives $1+\log_+(t^2/s^2)\le1+\log_+(2t^2)$.
By spectral calculus, \eqref{eq:cwikel-harmonic-tail} implies
\begin{equation}
\label{eq:cwikel-remainder-bound}
 \begin{aligned}
 \sum_{j>N_s}\|r_s(g_jf)u_j\|_H^2
 &=\sum_{j>N_s}\tau\bigl(r_s(g_jf)^2\bigr)\\
 &\le\tau\Bigg(f^2\sum_{j>N_s}g_j^2
                       \mathbf1_{(s,\infty)}(g_jf)\Bigg)\\
 &\le\tau\bigl(f^2(1+\log_+(2f^2))\bigr)=:Q(f).
 \end{aligned}
\end{equation}
It is easy to obtain, for $t\ge0$, the following scalar inequality
\[
 t(1+\log_+(2t))\le2\Phi(t). 
\]
Consequently, the Luxemburg normalization gives
$Q(f)\le2\tau(\Phi(f^2))\le2$.
Thus $R_s$, initially defined on $\mathcal P$ by \eqref{eq:R_s def on u_j},
extends to an operator $R_s\in\mathcal{L}_2(H)$ satisfying
\begin{equation}
  \label{eq:cwikel-Rs-L2-bound}
  \|R_s\|_{\mathcal{L}_2(H)}^2=\sum_{j\geq 1}\|R_su_j\|_H^2= \sum_{j>N_s}\|r_s(g_jf)u_j\|_H^2 \leq Q(f) \leq 2.
\end{equation}

On $\mathcal P$ we have
\begin{equation}
\label{eq:cwikel-decomposition}
 \lambda(f)G=F_s+Q_s+R_s.
\end{equation}
The right-hand side of \eqref{eq:cwikel-decomposition} is bounded.
Since it agrees with $\lambda(f)G$ on the dense subspace $\mathcal P$,
it defines a bounded extension $T$ independent of $s$.
We now show that $\lambda(f)G$ is defined on all $H$. Let $\xi\in H$ and choose
$\xi_k\in\mathcal P$ with $\xi_k\to\xi$. Then
$G\xi_k\to G\xi$ and $\lambda(f)G\xi_k=T\xi_k\to T\xi$.
Closedness of $\lambda(f)$ gives
$G\xi\in\operatorname{Dom}\lambda(f)$ and
$\lambda(f)G\xi=T\xi$. Hence the natural product is defined on all of
$H$ and equals $T$. 
For simplicity, we still use $\lambda(f)G$ to denote the extended operator $T$. 
Moreover, $F_s+R_s$ is compact and
\[
 \|\lambda(f)G-(F_s+R_s)\|=\|Q_s\|\le s.
\]
Letting $s\downarrow0$ proves compactness of $\lambda(f)G$ and, consequently,
of $Q_s=\lambda(f)G-F_s-R_s$ for each fixed $s>0$.

The inequalities in \eqref{eq:prelim-singular-inequalities}, the rank bound for $F_s$, and \eqref{eq:cwikel-Rs-L2-bound} give  
\[
 \begin{aligned}
 d_{\lambda(f)G}(2s)
 &\le\operatorname{rank}F_s+d_{Q_s}(s)+d_{R_s}(s)\\
 &\le N_s+s^{-2}\|R_s\|_{\mathcal{L}_2(H)}^2 \\
 &\le 3s^{-2}.
 \end{aligned}
\]
Here $d_{Q_s}(s)=0$ because $\|Q_s\|\le s$ and the distribution
function uses the open interval $(s,\infty)$.
By \eqref{eq:cwikel-counting-norm} we obtain
\begin{equation*}
 \|\lambda(f)G\|_{\mathcal{L}_{2,\infty}(H)}^2
 \leq \sup_{s>0}(2s)^2d_{\lambda(f)G}(2s) \leq 4\times3\le12.
\end{equation*}
Together with the initial normalization reduction, this proves the theorem.
\end{proof}

\section{Consequences of the abstract endpoint estimate}
\label{sec:consequences}

We apply the abstract endpoint estimate to quantum tori to obtain the
$L\log L$ Cwikel estimate and weak trace class bounds for pseudodifferential sandwiches $P_\sigma\lambda(x)P_\rho$. 
We then derive a critical Sobolev embedding and finite-codimensional weighted Sobolev estimates on quantum tori.

\subsection{The endpoint Cwikel estimate on $\mathbb{T}_\theta^d$}
\label{subsec:application-quantum-tori}

In this section, $d\geq2$, $\M=L_\infty(\TT_\theta^d)$ and $H=L_2(\mathbb T_\theta^d)$. We use the canonical trace
and Laplacian from the Section 2.3. In particular, $-\Delta\ge0$.
Write $\|\cdot\|_{p,\infty}$ for the quasi-norm in
$\mathcal{L}_{p,\infty}(H)=L_{p,\infty}(\mathcal B(H),\operatorname{Tr}_H)$.
The abstract estimate \eqref{eq:abstract-endpoint} gives the following quantum torus version of the
classical estimate in \cite[Theorem~1.1]{SZ22}.

\begin{theorem}
\label{thm:torus-endpoint}
Let $d\geq2$ and let $\theta$ be a real antisymmetric $d\times d$ matrix.
For every $x\in L\log L(\mathbb T_\theta^d)$, the sandwich operator
belongs to $\mathcal{L}_{1,\infty}(L_2(\mathbb T_\theta^d))$ and satisfies
\begin{equation}
\label{eq:torus-endpoint}
 \big\|(1-\Delta)^{-\frac{d}{4}}\lambda(x)(1-\Delta)^{-\frac{d}{4}}\big\|_{\mathcal{L}_{1,\infty}(H)}
 \le C_d\|x\|_{L\log L(\mathbb T_\theta^d)},
\end{equation}
where $C_d$ depends only on $d$ and is independent of $\theta$.
For unbounded $x$, this sandwich operator denotes the unique bounded operator determined by 
\begin{equation}
  \label{eq:torus-sandwich-form}
  \big\langle(1-\Delta)^{-\frac d4}\lambda(x)(1-\Delta)^{-\frac d4}\xi,\eta\big\rangle=\tau\Bigl(\bigl((1-\Delta)^{-\frac{d}{4}}\eta\bigr)^*
             x(1-\Delta)^{-\frac{d}{4}}\xi\Bigr)
\end{equation}
for Fourier polynomials $\xi,\eta$. If $x\ge0$, this operator is
positive.
\end{theorem}

\begin{proof}
The diagonal symbol of the operator $(1-\Delta)^{-d/4}$ in the unitary
Fourier basis $\{U^n\}_{n\in\mathbb Z^d}$ is
$a_n=(1+|n|^2)^{-d/4}$. Elementary lattice counting gives
\[
 \#\{n\in\mathbb Z^d:a_n>t\}\le6^dt^{-2}\quad(0<t<1),
\]
while the counting function vanishes for $t\ge1$. Consequently,
\[
 \kappa_d:=\|\{a_n\}_{n\in\mathbb Z^d}\|_{\ell_{2,\infty}(\mathbb Z^d)}
 \le6^{\frac{d}{2}}.
\]

Suppose first that $x\ge0$. Theorem~\ref{thm:abstract-endpoint}, applied
to $f=x^{1/2}$ and $G=(1-\Delta)^{-d/4}$, shows that
$\lambda(x^{1/2})(1-\Delta)^{-d/4}\in\mathcal{L}_{2,\infty}(H)$ and gives
\begin{equation}
\label{eq:torus-square-root}
 \|\lambda\bigl(x^{\frac{1}{2}}\bigr)(1-\Delta)^{-\frac{d}{4}}\|_{2,\infty}^2
 \le12\kappa_d^2\Norm{x}_{L\log L(\mathbb T_\theta^d)}.
\end{equation}
Hence the sandwich satisfies 
\begin{equation*}
  \begin{split}
    \Norm{(1-\Delta)^{-\frac{d}{4}}\lambda(x)(1-\Delta)^{-\frac{d}{4}}}_{1,\infty}
    &=\Norm{\Bigl(\lambda\bigl(x^{\frac{1}{2}}\bigr)(1-\Delta)^{-\frac{d}{4}}\Bigr)^*\lambda\bigl(x^{\frac{1}{2}}\bigr)(1-\Delta)^{-\frac{d}{4}}}_{1,\infty} \\
    &=\Norm{\lambda\bigl(x^{\frac{1}{2}}\bigr)(1-\Delta)^{-\frac{d}{4}}}_{2,\infty}^2 \\
    &\leq 12\kappa_d^2\Norm{x}_{L\log L}. 
  \end{split}
\end{equation*}

For general $x$, let $x=v|x|$ be its polar decomposition, and set
\[
 A=\lambda\bigl(|x|^{\frac{1}{2}}\bigr)(1-\Delta)^{-\frac{d}{4}},\qquad
 B=\lambda\bigl(|x^*|^{\frac{1}{2}}\bigr)(1-\Delta)^{-\frac{d}{4}}.
\]
The positive operators $|x|$ and $|x^*|$ have the same
$\tau$-distribution, so their $L\log L$ norms both equal
$\|x\|_{L\log L}$.
Applying \eqref{eq:torus-square-root} gives
\[
 \|A\|_{2,\infty}^2,\ \|B\|_{2,\infty}^2 \leq 12\kappa_d^2 \Norm{x}_{L\log L}.
\]
Since $|x^*|^{1/2}v|x|^{1/2}=x$, the operator
$B^*\lambda(v)A$ satisfies \eqref{eq:torus-sandwich-form} on Fourier polynomials.
The weak form of H\"older's inequality \eqref{eq:prelim-weak-holde} and $\|\lambda(v)\|\le1$ yield
\[
 \|B^*\lambda(v)A\|_{1,\infty}
 \le2\|B\|_{2,\infty}\|A\|_{2,\infty}
 \leq 24\kappa_d^2 \Norm{x}_{L\log L(\mathbb{T}_\theta^d)}.
\]
Thus \eqref{eq:torus-endpoint} holds with $C_d=24\kappa_d^2$.
Finally, density of the Fourier polynomials in $H$ gives uniqueness of its bounded operator satisfying \eqref{eq:torus-sandwich-form}.
\end{proof}

\begin{corollary}
\label{cor:pdo-sandwich-endpoint}
Let $\rho,\sigma\in S^{-d/2}(\mathbb R^d;\mathcal{S}(\mathbb T_\theta^d))$.
For every $x\in L\log L(\mathbb T_\theta^d)$, the operator
$P_\sigma^*\lambda(x)P_\rho$ belongs to $\mathcal{L}_{1,\infty}(H)$, where
$H=L_2(\mathbb T_\theta^d)$, and
\begin{equation}
\label{eq:pdo-sandwich-endpoint}
 \begin{aligned}
 \|P_\sigma^*\lambda(x)P_\rho\|_{1,\infty}
 &\le C_d\|B_\sigma\|_{\mathcal B(H)}\|B_\rho\|_{\mathcal B(H)}
         \|x\|_{L\log L(\mathbb T_\theta^d)}.
 \end{aligned}
\end{equation}
Here $B_\rho=(1-\Delta)^{d/4}P_\rho$, $B_\sigma=(1-\Delta)^{d/4}P_\sigma$, 
$C_d$ is the constant in Theorem~\ref{thm:torus-endpoint}, and the
sandwich denotes the unique bounded operator representing the form
\begin{equation}
  \label{eq:pdo-sandwich-form}
  (\xi,\eta)\longmapsto\tau\bigl((P_\sigma\eta)^*xP_\rho\xi\bigr),
 \qquad \xi,\eta\in \mathcal{S}(\mathbb T_\theta^d).
\end{equation}
In particular, this operator is positive when $x\ge0$ and $\rho=\sigma$.
\end{corollary}

\begin{proof}
  Equation \eqref{eq:prelim-pdo-sobolev-mapping}, with $m=-d/2$ and $s=d/2$, gives bounded operators $P_\rho,P_\sigma:H\to W_2^{d/2}(\mathbb T_\theta^d)$. Thus 
$B_\rho=(1-\Delta)^{d/4}P_\rho$ and $B_\sigma=(1-\Delta)^{d/4}P_\sigma$
are bounded on $H$. 
Theorem~\ref{thm:torus-endpoint} supplies the bounded operator $(1-\Delta)^{-d/4}\lambda(x)(1-\Delta)^{-d/4}$. 
Its matrix-coefficient identity \eqref{eq:torus-sandwich-form} extends to $\mathcal{S}(\TT_\theta^d)$ by \eqref{eq:prelim-smooth-fourier-density}.
Indeed, $x\in L_1$, and convergence in seminorms \eqref{eq:prelim-smooth-seminorms} implies convergence in operator norm.
Since $B_\rho$ and $B_\sigma$ preserve $\S(\TT_\theta^d)$, the operator 
\[
 P_\sigma^*\lambda(x)P_\rho=B_\sigma^*(1-\Delta)^{-\frac{d}{4}}\lambda(x)(1-\Delta)^{-\frac{d}{4}}B_\rho 
\]
represent the form in \eqref{eq:pdo-sandwich-form}.
The ideal property \eqref{eq:prelim-ideal-property} now yields
\[
 \begin{aligned}
 &\|B_\sigma^*(1-\Delta)^{-\frac{d}{4}}\lambda(x)(1-\Delta)^{-\frac{d}{4}}B_\rho\|_{1,\infty}\\
 &\qquad\le\|B_\sigma\|_{\mathcal B(H)}
       \|(1-\Delta)^{-\frac{d}{4}}\lambda(x)(1-\Delta)^{-\frac{d}{4}}\|_{1,\infty}
       \|B_\rho\|_{\mathcal B(H)},
 \end{aligned}
\]
which proves \eqref{eq:pdo-sandwich-endpoint}.
\end{proof}

\subsection{A critical Sobolev embedding theorem}
\par The preceding subsection shows that, on quantum tori, the sandwich
operator $(1-\Delta)^{-d/4}\lambda(x)(1-\Delta)^{-d/4}$ belongs to
$\mathcal{L}_{1,\infty}$ and satisfies the estimate \eqref{eq:torus-endpoint}. Its operator norm is therefore controlled by
$\Norm{x}_{L\log L}$. Conbined with K\"othe duality, this gives a critical Sobolev embedding. 
\par We define the Orlicz space $\exp(L_p)(\M)$, $1\leq p<\infty$, via the Orlicz function $\Psi_p(t)=\exp(t^p)-1$. 
By noncommutative K\"othe duality, $\exp(L_1)(\M)$ identifies with the dual of $L\log L(\M)$, up to equivalence of norms; see \cite[Theorem 5.6]{DoddsDoddsdePagter1993}. 
\begin{corollary}
  \label{cor:critical sobolev}
  For $x\in W_2^{d/2}(\mathbb{T}_\theta^d)$, we have $x\in \exp(L_2)(\mathbb{T}_\theta^d)$, and moreover, 
  \begin{equation}
    \label{eq:Sobolev expL_2}
    \Norm{x}_{\exp(L_2)(\mathbb{T}_\theta^d)}\leq C_d'\Norm{x}_{W_2^{d/2}(\mathbb{T}_\theta^d)}. 
  \end{equation}
\end{corollary}
\begin{proof}
  To prove \eqref{eq:Sobolev expL_2}, it suffices to show that for every $y\in L_2(\mathbb{T}_\theta^d)$ we have 
  \begin{equation*}
    \Norm{(1-\Delta)^{-\frac{d}{4}}y}_{\exp(L_2)(\mathbb{T}_\theta^d)}\leq C_d'\Norm{y}_{L_2(\mathbb{T}_\theta^d)}. 
  \end{equation*}
  By the duality between $\exp(L_1)$ and $L\log L$ spaces (up to an equivalent norm), for $x\geq0$, we have 
  \begin{equation}
    \label{eq:orlicz-duality-LlogL}
    \Norm{x}_{\exp(L_1)(\mathbb{T}_\theta^d)}
      \asymp\sup\left\{\tau(xz):z\geq0, \Norm{z}_{L\log L(\TT_\theta^d)}\leq1\right\}. 
  \end{equation}
  Therefore, plugging $x=\big|[(1-\Delta)^{-d/4}y]^*\big|^2$ into \eqref{eq:orlicz-duality-LlogL} implies
  \begin{equation*}
    \begin{split}
      \Norm{(1-\Delta)^{-\frac{d}{4}}y}_{\exp(L_2)(\mathbb{T}_\theta^d)}^2
      &=\Norm{[(1-\Delta)^{-\frac{d}{4}}y]^*}_{\exp(L_2)(\mathbb{T}_\theta^d)}^2 \\ 
      &=\Norm{\big|[(1-\Delta)^{-\frac{d}{4}}y]^*\big|^2}_{\exp(L_1)(\mathbb{T}_\theta^d)} \\ 
      &\leq C_{\text{eq}} \sup\left\{\tau\Bigl(\big|[(1-\Delta)^{-\frac{d}{4}}y]^*\big|^2z\Bigr):\Norm{z}_{L\log L}\leq1\right\} \\
      &=C_{\text{eq}} \sup\left\{\left|\inner{y,(1-\Delta)^{-\frac{d}{4}}\lambda(z)(1-\Delta)^{-\frac{d}{4}}y}\right|:\Norm{z}_{L\log L}\leq1\right\} \\
      &\leq C_{\text{eq}}C_d \Norm{y}_{L_2(\mathbb{T}_\theta^d)}^2,
    \end{split}
  \end{equation*}
  where the last inequality follows from Theorem~\ref{thm:torus-endpoint}, which gives 
  \begin{equation*}
    \big\|{(1-\Delta)^{-\frac{d}{4}}\lambda(z)(1-\Delta)^{-\frac{d}{4}}}\big\|_\infty \leq \big\|{(1-\Delta)^{-\frac{d}{4}}\lambda(z)(1-\Delta)^{-\frac{d}{4}}}\big\|_{1,\infty} \leq C_d\Norm{z}_{L\log L}\leq C_d. 
  \end{equation*}
  This proves the corollary with $C'_d=(C_{\text{eq}}C_d)^{1/2}$. 
\end{proof}

\begin{remark}
The critical Sobolev embedding in Corollary~\ref{cor:critical sobolev} is equivalent, up to constants, to the uniform operator norm estimate
\begin{equation*}
 \big\|(1-\Delta)^{-\frac d4}\lambda(x)(1-\Delta)^{-\frac d4}\big\|_{\infty}
 \le C_d\|x\|_{L\log L(\mathbb T_\theta^d)},
 \qquad x\in L\log L(\mathbb T_\theta^d)_+.
\end{equation*}
The embedding implies this estimate by Orlicz H\"older's inequality, while the converse follows from the K\"othe duality argument above.
For the corresponding equivalence in the classical setting, see \cite[Theorem~2.3]{LSZ20Kalton}.
\end{remark}

\subsection{Finite-codimensional weighted Sobolev estimates}
\label{subsec:further-consequences}

\par The decomposition in \eqref{eq:cwikel-decomposition} also gives
subspaces of finite codimension at most $n$ on which $\lambda(f)G$ satisfies the estimate below. 
\par As in Section~\ref{sec:cwikel-estimates}, we assume that the von Neumann
algebra $(\M,\tau)$ is finite and that $H=L_2(\M)$ admits a basis of
unitaries in $\M$.

\begin{corollary}
\label{cor:finite-codimension}
Suppose $f\in L_0(\M)_+$ satisfies $f^2\in L\log L(\M)$, and
$G\in\mathcal B(H)$ is given by $Gu_j=g_ju_j$, with
$g=\{g_j\}\in\ell_{2,\infty}$.
Then for every integer $n\ge1$, there is a closed subspace $E_n\subset H$
of codimension at most $n$ such that
\begin{equation}
\label{eq:finite-codimension-bound}
 \|\lambda(f)G\xi\|_H^2
 \le\frac{12}{n}\Norm{f^2}_{L\log L(\M)}\Norm{g}_{\ell_{2,\infty}}^2\|\xi\|_H^2,
 \qquad \xi\in E_n.
\end{equation}
\end{corollary}

\begin{proof}
If $f=0$ or $g=0$, take $E_n=H$. Otherwise, first assume that the normalization condition
\eqref{eq:cwikel-normalization} holds.  
We next construct $E_n$ with codim$(E_n)\leq n$ such that
\begin{equation}
  \label{eq:finite-codimension-bound-normalized}
  \|\lambda(f)G\xi\|_H^2
 \le\frac{12}{n}\|\xi\|_H^2, \qquad \xi\in E_n. 
\end{equation}
Choose $s=\sqrt{3/n}>0$. For this $s$ we use the decomposition $\lambda(f)G=F_s+Q_s+R_s$ from \eqref{eq:cwikel-decomposition}. 
Recall that $F_s=\lambda(f)GP_s$. We set 
\begin{equation}
  \label{eq:finite-codimension-subspace}
  \Pi_s=\mathbf1_{(s,\infty)}(|R_s|),
 \qquad E_n=\ker P_s\cap\ker\Pi_s.
\end{equation}
The Hilbert--Schmidt bound \eqref{eq:cwikel-Rs-L2-bound} implies
\[
 \operatorname{codim}_H E_n
 \le\operatorname{rank}P_s+\operatorname{rank}\Pi_s
 \le s^{-2}+s^{-2}\|R_s\|_{\mathcal{L}_2(H)}^2
 \le 3s^{-2}=n.
\]
For $\xi\in E_n$ we have $F_s\xi=0$ and
$\|R_s\xi\|_H\le s\|\xi\|_H$. Since $\|Q_s\|\le s$, it follows that
\begin{equation*}
  \|\lambda(f)G\xi\|_H\le2s\|\xi\|_H=2\sqrt{\frac 3n}\|\xi\|_H. 
\end{equation*}
This is exactly \eqref{eq:finite-codimension-bound-normalized}.

For general $f$ and $G$, put
$a=\|f^2\|_{L\log L(\M)}^{1/2}$, $b=\|g\|_{\ell_{2,\infty}}$,
and write $G=|G|V$ with the diagonal unitary $V$ from the proof of
Theorem~\ref{thm:abstract-endpoint}. Apply the normalized construction
to $\tilde f=f/a$ and $\tilde G=|G|/b$, obtaining a subspace
$\widetilde E_n$. Then $E_n=V^*\widetilde E_n$ is closed and has
codimension at most $n$. The identity
$\lambda(f)G=ab(\lambda(\tilde f)\tilde G)V$ gives
\eqref{eq:finite-codimension-bound}.
\end{proof}

On the quantum tori, Corollary \ref{cor:finite-codimension} yields 
the following weighted estimate on subspaces of finite codimension in $W_2^{d/2}(\mathbb T_\theta^d)$. 
For $x\ge0$ and $u\in\operatorname{Dom}\lambda(x^{1/2})$, we use $\tau(u^*xu)$ to denote $\|\lambda(x^{1/2})u\|_2^2$.

\begin{corollary}
\label{cor:sobolev-subspace}
Let $x\in L\log L(\mathbb T_\theta^d)_+$ and let $n\ge1$ be an integer.
There exists a closed subspace
$F_n^{\mathrm{Sob}}\subset W_2^{d/2}(\mathbb T_\theta^d)$ of codimension at most $n$
such that
\begin{equation*}
 \tau(u^*xu) \le\frac{c_d}{n}\|x\|_{L\log L(\mathbb T_\theta^d)} \|u\|_{W_2^{d/2}(\mathbb T_\theta^d)}^2,
 \qquad u\in F_n^{\mathrm{Sob}}.
\end{equation*}
\end{corollary}

\begin{proof}
The case $x=0$ is immediate, with $F_n^{\mathrm{Sob}}=W_2^{d/2}(\mathbb T_\theta^d)$.
Otherwise, apply Corollary~\ref{cor:finite-codimension} with 
\begin{equation*}
  \M=\TT_\theta^d, \qquad f=x^{\frac{1}{2}}, \qquad G=(1-\Delta)^{-\frac{d}{4}}.
\end{equation*}
For each $n$, this gives a subspace $E_n$ with $\operatorname{codim}(E_n)\leq n$. Set 
\begin{equation*}
  F_n^{\mathrm{Sob}}=(1-\Delta)^{-\frac{d}{4}} E_n\subset W_2^{\frac{d}{2}}(\mathbb{T}_\theta^d). 
\end{equation*}
By \eqref{eq:prelim-sobolev}, $(1-\Delta)^{-d/4}$ is an isometric isomorphism from
$L_2(\mathbb T_\theta^d)$ onto $W_2^{d/2}(\mathbb T_\theta^d)$,
hence $F_n^{\mathrm{Sob}}$ is closed in the Sobolev norm and has the same codimension as $E_n$.
Moreover, for $u\in F_n^{\mathrm{Sob}}$, we can write $u=(1-\Delta)^{-d/4}\xi$, where $\xi\in E_n$. By Corollary~\ref{cor:finite-codimension}, 
\begin{equation*}
  \begin{split}
  \tau(u^*xu)&=\Norm{x^{\frac{1}{2}}u}_{L_2(\mathbb{T}_\theta^d)}^2=\Norm{\lambda\bigl(x^{\frac{1}{2}}\bigr)(1-\Delta)^{-\frac{d}{4}}\xi}_2^2 \\
  &\leq \frac{12\kappa_d^2}{n}\Norm{x}_{L\log L(\mathbb{T}_\theta^d)}\Norm{\xi}_2^2 \\
  &\leq \frac{c_d}{n}\Norm{x}_{L\log L(\mathbb{T}_\theta^d)}\Norm{u}_{W_2^{d/2}(\mathbb{T}_\theta^d)}^2, 
  \end{split}
\end{equation*}
where $c_d=12\kappa_d^2$.
\end{proof}

\begin{remark}
  Corollary~\ref{cor:sobolev-subspace} is a quantum torus counterpart
of \eqref{eq:Solomyak approximation}.
In the classical case, the proof in \cite{SZ22} uses local polynomial
approximation and a Besicovitch covering.
Here $E_n$ is the intersection of the two kernels in
\eqref{eq:finite-codimension-subspace}.
Its codimension is bounded by the ranks of the two projections.
\end{remark}

\section{Weyl's laws and Connes' integration formulas for Orlicz potentials}
\label{sec:orlicz-weyl}
In this section, we prove Weyl's laws for the sandwiches
in \eqref{eq:ow-operator-and-symbol} with $L\log L$ potentials.
We use these asymptotics to derive Connes' integration formulas
for every continuous normalized trace.
We also prove a semiclassical Weyl's law for critical
Schr\"odinger operators.

Throughout this section, $d\geq2$ and $H=L_2(\mathbb T_\theta^d)$.
We use the algebra $\Pi$ and the symbol map from
Section~\ref{subsec:principal-symbol-algebra}.
For $x\in L\log L(\mathbb T_\theta^d)$, the operator
\[
 (1-\Delta)^{-\frac{d}{4}}\lambda(x)(1-\Delta)^{-\frac{d}{4}}
\]
is defined by the form in Theorem~\ref{thm:torus-endpoint}.
That theorem gives
\begin{equation}
\label{eq:ow-endpoint-continuity}
 \bigl\|(1-\Delta)^{-\frac{d}{4}}\lambda(x-y)(1-\Delta)^{-\frac{d}{4}}\bigr\|_{1,\infty}
 \leq C_d\|x-y\|_{L\log L(\mathbb T_\theta^d)},
 \qquad x,y\in L\log L(\mathbb T_\theta^d).
\end{equation}
The map $x\mapsto(1-\Delta)^{-d/4}\lambda(x)(1-\Delta)^{-d/4}$
is linear and preserves adjoints.
It maps positive elements to positive operators.

\subsection{Spectral preliminaries and approximation in $L\log L$}

We recall the following results from \cite{Z26}.

\begin{proposition}
\label{prop:ow-bounded-input}
For every $A\in\Pi$,
\begin{equation*}
 \lim_{j\to\infty}(j+1)\mu\bigl(j,A(1-\Delta)^{-\frac{d}{2}}\bigr)
 =\frac{1}{d}\int_{\mathbb S^{d-1}}
             \tau\bigl(|\operatorname{sym}(A)(s)|\bigr)\,ds,
\end{equation*}
and
\begin{equation*}
 [(1-\Delta)^{-\frac{d}{4}},A]\in(\mathcal{L}_{2,\infty}(H))_0.
\end{equation*}
If $A=A^*$, then
\begin{equation*}
 \lim_{j\to\infty}(j+1)\lambda_j^\pm\bigl((1-\Delta)^{-\frac{d}{4}}A(1-\Delta)^{-\frac{d}{4}}\bigr)
 =\frac{1}{d}\int_{\mathbb S^{d-1}}
             \tau\bigl(\operatorname{sym}(A)(s)_\pm\bigr)\,ds.
\end{equation*}
\end{proposition}
\begin{proof}
The singular-value formula follows from
\cite[Corollary~5.2]{Z26} with $p=1$.
The commutator inclusion follows from
\cite[equation~(5.2)]{Z26} with $\alpha=d/2$.
The eigenvalue formulas follow from
\cite[Theorem~5.8]{Z26} with $m=d$.
\end{proof}

\begin{lemma}
\label{lem:ow-smooth-density}
Fourier polynomials are dense in $L\log L(\mathbb T_\theta^d)$.
Every selfadjoint element of $L\log L(\mathbb T_\theta^d)$ can be approximated by selfadjoint Fourier
polynomials.
\end{lemma}

\begin{proof}
Let $x\in L\log L(\mathbb T_\theta^d)$.
Since $\Phi(2t)\leq4\Phi(t)$, we have
$\tau(\Phi(\frac{|x|}{\varepsilon}))<\infty$
for every $\varepsilon>0$.
For the polar decomposition $x=v|x|$, set
\[
 x^{(M)}=v|x|\mathbf1_{[0,M]}(|x|).
\]
Dominated convergence in the spectral calculus of $|x|$ gives
\[
 \tau\Bigl(\Phi\Bigl(\frac{|x-x^{(M)}|}{\varepsilon}\Bigr)\Bigr)\longrightarrow0.
\]
Thus $x^{(M)}\to x$ in the Luxemburg norm.

For bounded $y$, define its product Fej\'er means by
\begin{equation}
\label{eq:ow-fejer-means}
 \begin{aligned}
 w_N(n)&=\prod_{k=1}^d\Bigl(1-\frac{|n_k|}{N+1}\Bigr)_+,\\
 F_Ny&=\sum_{n\in\mathbb Z^d}w_N(n)\widehat y(n)U^n.
 \end{aligned}
\end{equation}
See \cite[Section~3]{CXY13}.
Each $F_N$ averages the canonical torus action against a nonnegative
kernel of integral one. Consequently,
$\|F_Ny\|_\infty\leq\|y\|_\infty$, and $F_N$ preserves adjoints.
Since $y$ is bounded and $\tau(1)=1$, we have $y\in H$.
Parseval's identity~\eqref{eq:prelim-parseval} gives
\begin{equation}
\label{eq:ow-fejer-l2-convergence}
 \|F_Ny-y\|_2^2
 =\sum_{n\in\mathbb Z^d}|w_N(n)-1|^2|\widehat y(n)|^2
 \longrightarrow0.
\end{equation}
The limit follows from dominated convergence:
$0\leq w_N(n)\leq1$ and $w_N(n)\to1$ for each $n$.
See also \cite[Proposition~3.1]{CXY13}.
Since $\|F_Ny-y\|_1\leq\|F_Ny-y\|_2$,
we also have $\|F_Ny-y\|_1\to0$.
For every $\varepsilon>0$,
\[
 \tau\Bigl(\Phi\Bigl(\frac{|F_Ny-y|}{\varepsilon}\Bigr)\Bigr)
 \leq\frac{1}{\varepsilon}
       \log\Bigl(e+\frac{2\|y\|_\infty}{\varepsilon}\Bigr)
       \|F_Ny-y\|_1\longrightarrow0.
\]
It follows that $F_Ny\to y$ in $L\log L(\mathbb T_\theta^d)$.
Applying the Fej\'er approximation to $x^{(M)}$ proves density.
Both the truncation and the Fej\'er means preserve selfadjointness.
\end{proof}

\subsection{Weyl's laws with Orlicz potentials}

\begin{theorem}
\label{thm:ow-sandwich-weyl}
Let $S,T\in\Pi$, $x\in L\log L(\mathbb T_\theta^d)$, and
$R\in(\mathcal{L}_{1,\infty}(H))_0$. Set
\begin{equation}
\label{eq:ow-operator-and-symbol}
 \begin{aligned}
 K&=S(1-\Delta)^{-\frac{d}{4}}\lambda(x)(1-\Delta)^{-\frac{d}{4}}T+R,\\
 a(s)&=\operatorname{sym}(S)(s)x\operatorname{sym}(T)(s).
 \end{aligned}
\end{equation}
Then $K\in\mathcal{L}_{1,\infty}(H)$,
$a\in C(\mathbb S^{d-1};L_1(\mathbb T_\theta^d))$, and
\begin{equation}
\label{eq:ow-sandwich-singular}
 \lim_{j\to\infty}(j+1)\mu(j,K)
 =\frac{1}{d}\int_{\mathbb S^{d-1}}\tau(|a(s)|)\,ds.
\end{equation}
Suppose $K_1$ and $K_2$ have representations \eqref{eq:ow-operator-and-symbol},
with functions $a_1$ and $a_2$.
If $K_1-K_2\in(\mathcal L_{1,\infty}(H))_0$,
then $a_1=a_2$ almost everywhere.
We call $a$ the generalized principal symbol of order $-d$
of $K$ on $\mathbb S^{d-1}$.
If $K=K^*$, then $a=a^*$ almost everywhere and
\begin{equation}
\label{eq:ow-sandwich-signed}
 \lim_{j\to\infty}(j+1)\lambda_j^\pm(K)
 =\frac{1}{d}\int_{\mathbb S^{d-1}}\tau(a(s)_\pm)\,ds.
\end{equation}

The same conclusions hold for
\[
 K=S(-\Delta)^{-\frac{d}{4}}
 \lambda(x)(-\Delta)^{-\frac{d}{4}}T+R,
\]
with the same generalized principal symbol $a$.
Here $(-\Delta)^{-\frac{d}{4}}$ is the generalized negative power,
defined to vanish on the constant subspace.
\end{theorem}

\begin{proof}
Membership of $K$ in $\mathcal L_{1,\infty}$ follows from
Theorem~\ref{thm:torus-endpoint} and the ideal property
\eqref{eq:prelim-ideal-property}.
Both $\operatorname{sym}(S)$ and $\operatorname{sym}(T)$
are norm-continuous and bounded.
Since $x\in L_1(\mathbb T_\theta^d)$, the function $a$ is $L_1$-norm continuous, and
\[
 \int_{\mathbb S^{d-1}}\tau(|a(s)|)\,ds
 \leq |\mathbb S^{d-1}|\,\|S\|\,\|x\|_1\,\|T\|.
\]

Choose Fourier polynomials $x_n\to x$ in $L\log L(\mathbb T_\theta^d)$, and put
\[
 \begin{aligned}
 K_n&=S(1-\Delta)^{-\frac{d}{4}}\lambda(x_n)(1-\Delta)^{-\frac{d}{4}}T+R,\\
 A_n&=S\lambda(x_n)T\in\Pi,
 \qquad a_n=\operatorname{sym}(A_n).
 \end{aligned}
\]
By \eqref{eq:prelim-ideal-property} and \eqref{eq:ow-endpoint-continuity},
\[
\begin{aligned}
 \|K_n-K\|_{1,\infty}
 &\leq \|S\|\,\|T\|\,
 \bigl\|(1-\Delta)^{-d/4}\lambda(x_n-x)
                     (1-\Delta)^{-d/4}\bigr\|_{1,\infty}\\
 &\leq C_d\|S\|\,\|T\|\,
       \|x_n-x\|_{L\log L(\mathbb T_\theta^d)}
 \longrightarrow0.
\end{aligned}
\]
 Moreover,
\begin{equation}
\label{eq:ow-symbol-convergence}
 \int_{\mathbb S^{d-1}}\tau(|a_n(s)-a(s)|)\,ds
 \leq |\mathbb S^{d-1}|\,\|S\|\,\|T\|\,\|x_n-x\|_1
 \longrightarrow0.
\end{equation}

For $y\in\mathcal S(\mathbb T_\theta^d)$,
\[
 \begin{aligned}
 &S(1-\Delta)^{-\frac{d}{4}}\lambda(y)(1-\Delta)^{-\frac{d}{4}}T-S\lambda(y)T(1-\Delta)^{-\frac{d}{2}}\\
 &\qquad=S[(1-\Delta)^{-\frac{d}{4}},\lambda(y)](1-\Delta)^{-\frac{d}{4}}T
          +S\lambda(y)[(1-\Delta)^{-\frac{d}{2}},T].
 \end{aligned}
\]
By the commutator inclusion in
Proposition~\ref{prop:ow-bounded-input} and
\eqref{eq:prelim-small-product},
\[
 [(1-\Delta)^{-\frac{d}{4}},\lambda(y)](1-\Delta)^{-\frac{d}{4}}\in(\mathcal{L}_{1,\infty}(H))_0.
\]
Similarly, expanding the commutator gives
\[
 \begin{aligned}
 \bigl[(1-\Delta)^{-\frac{d}{2}},T\bigr]
 &=(1-\Delta)^{-\frac{d}{4}}[(1-\Delta)^{-\frac{d}{4}},T]+[(1-\Delta)^{-\frac{d}{4}},T](1-\Delta)^{-\frac{d}{4}}\\
 &\in(\mathcal{L}_{1,\infty}(H))_0.
 \end{aligned}
\]
Therefore
\begin{equation}
\label{eq:ow-normal-form}
 K_n-A_n(1-\Delta)^{-\frac{d}{2}}\in(\mathcal{L}_{1,\infty}(H))_0.
\end{equation}
Proposition~\ref{prop:ow-bounded-input} and
Lemma~\ref{lem:weak-weyl-stability} give
\[
 \lim_{j\to\infty}(j+1)\mu(j,K_n)
 =\frac{1}{d}\int_{\mathbb S^{d-1}}\tau(|a_n(s)|)\,ds.
\]
By \eqref{eq:ow-symbol-convergence}, the right-hand side converges to
$\frac{1}{d}\int_{\mathbb S^{d-1}}\tau(|a(s)|)\,ds$.
Lemma~\ref{lem:weak-ideal-approximation} now gives
\eqref{eq:ow-sandwich-singular}.

For each of the two representations, approximate the potential by Fourier
polynomials and construct operators $K_{i,n}$ and
$A_{i,n}\in\Pi$, $i=1,2$.
Set
\[
 a_{i,n}=\operatorname{sym}(A_{i,n}),\qquad i=1,2.
\]
Applying \eqref{eq:ow-normal-form} to both representations gives
\[
 K_{1,n}-K_{2,n}
 -(A_{1,n}-A_{2,n})(1-\Delta)^{-d/2}
 \in(\mathcal L_{1,\infty}(H))_0.
\]
Proposition~\ref{prop:ow-bounded-input} and
Lemma~\ref{lem:weak-weyl-stability} therefore give
\[
 \lim_{j\to\infty}(j+1)\mu(j,K_{1,n}-K_{2,n})
 =\frac1d\int_{\mathbb S^{d-1}}\tau(|a_{1,n}(s)-a_{2,n}(s)|)\,ds.
\]
By \eqref{eq:ow-symbol-convergence}, the right-hand side converges to
$\frac1d\int_{\mathbb S^{d-1}}\tau(|a_1(s)-a_2(s)|)\,ds$.
Since $K_{1,n}-K_{2,n}\to K_1-K_2$ in $\mathcal L_{1,\infty}$,
Lemma~\ref{lem:weak-ideal-approximation} gives
\[
 \frac1d\int_{\mathbb S^{d-1}}\tau(|a_1(s)-a_2(s)|)\,ds
 =\lim_{j\to\infty}(j+1)\mu(j,K_1-K_2)=0.
\]
Hence $a_1=a_2$ almost everywhere.
Taking adjoints in \eqref{eq:ow-operator-and-symbol} gives
\[
 K^*=T^*(1-\Delta)^{-d/4}\lambda(x^*)
                  (1-\Delta)^{-d/4}S^*+R^*.
\]
The symbol map preserves adjoints, so the associated function is
\[
 \operatorname{sym}(T^*)(s)x^*
 \operatorname{sym}(S^*)(s)=a(s)^*.
\]
Uniqueness therefore implies $a=a^*$ almost everywhere
when $K=K^*$.

Suppose now that $K=K^*$. Then
$\|\operatorname{Re}K_n-K\|_{1,\infty}\to0$. Since
\[
 \begin{aligned}
 A_n(1-\Delta)^{-\frac{d}{2}}-(1-\Delta)^{-\frac{d}{4}}A_n(1-\Delta)^{-\frac{d}{4}}
 &=[A_n,(1-\Delta)^{-\frac{d}{4}}](1-\Delta)^{-\frac{d}{4}}\\
 &\in(\mathcal{L}_{1,\infty}(H))_0,
 \end{aligned}
\]
equation~\eqref{eq:ow-normal-form} implies
\begin{equation}
\label{eq:ow-selfadjoint-normal-form}
 \operatorname{Re}K_n-(1-\Delta)^{-\frac d4}
   (\operatorname{Re}A_n)(1-\Delta)^{-\frac d4}
 \in(\mathcal L_{1,\infty}(H))_0.
\end{equation}
The difference in \eqref{eq:ow-selfadjoint-normal-form} is selfadjoint.
Thus the eigenvalue formulas in
Proposition~\ref{prop:ow-bounded-input} and
Lemma~\ref{lem:ow-signed-stability} yield
\[
 \lim_{j\to\infty}(j+1)\lambda_j^\pm(\operatorname{Re}K_n)
 =\frac{1}{d}\int_{\mathbb S^{d-1}}
       \tau\bigl((\operatorname{Re}a_n(s))_\pm\bigr)\,ds.
\]
For selfadjoint $b,c\in L_1(\mathbb T_\theta^d)$,
\begin{equation}
\label{eq:ow-traced-positive-parts}
 \tau(b_\pm)=\frac{\|b\|_1\pm\tau(b)}{2},
 \qquad |\tau(b_\pm)-\tau(c_\pm)|\leq\|b-c\|_1.
\end{equation}
Equations~\eqref{eq:ow-symbol-convergence} and
\eqref{eq:ow-traced-positive-parts} give
\[
\int_{\mathbb S^{d-1}}\tau((\operatorname{Re}a_n(s))_\pm)\,ds
\longrightarrow
\int_{\mathbb S^{d-1}}\tau(a(s)_\pm)\,ds.
\]
Applying Lemma~\ref{lem:ow-signed-stability} once more, we obtain
\eqref{eq:ow-sandwich-signed}.

To obtain the Riesz-potential version, consider the bounded
selfadjoint Fourier multiplier
\[
 B=(1-\Delta)^{\frac{d}{4}}(-\Delta)^{-\frac{d}{4}}.
\]
Its multiplier is zero at $n=0$ and equals
\[
 \Bigl(\frac{1+|n|^2}{|n|^2}\Bigr)^{\frac{d}{4}},
 \qquad n\neq0.
\]
Thus $\|B\|\leq2^{\frac{d}{4}}$ and $B-1$ is compact.
Since $B$ commutes with the Bessel potential, the form identity
\begin{equation}
\label{eq:ow-riesz-bessel-form}
 (-\Delta)^{-\frac{d}{4}}
   \lambda(x)(-\Delta)^{-\frac{d}{4}}
 =B(1-\Delta)^{-\frac{d}{4}}
   \lambda(x)(1-\Delta)^{-\frac{d}{4}}B
\end{equation}
and Theorem~\ref{thm:torus-endpoint} give a bounded extension with
\[
 \bigl\|(-\Delta)^{-\frac{d}{4}}
          \lambda(x)(-\Delta)^{-\frac{d}{4}}\bigr\|_{1,\infty}
 \leq2^{\frac{d}{2}}C_d\|x\|_{L\log L(\mathbb T_\theta^d)}.
\]
Multiplication on either side by a compact operator maps
$\mathcal{L}_{1,\infty}(H)$ into
$(\mathcal{L}_{1,\infty}(H))_0$, as follows by finite-rank
approximation of the compact factor in operator norm.
Expanding $B=1+(B-1)$ in \eqref{eq:ow-riesz-bessel-form} gives
\begin{equation}
\label{eq:ow-riesz-bessel-difference}
 \begin{aligned}
&(-\Delta)^{-\frac{d}{4}}
   \lambda(x)(-\Delta)^{-\frac{d}{4}}
 -(1-\Delta)^{-\frac{d}{4}}
   \lambda(x)(1-\Delta)^{-\frac{d}{4}}\\
 =&(B-1)(1-\Delta)^{-\frac{d}{4}}
   \lambda(x)(1-\Delta)^{-\frac{d}{4}}B+(1-\Delta)^{-\frac{d}{4}}
   \lambda(x)(1-\Delta)^{-\frac{d}{4}}(B-1)\\
 \in &(\mathcal L_{1,\infty}(H))_0.
 \end{aligned}
\end{equation}

Let
\[
 K'=S(-\Delta)^{-\frac d4}\lambda(x)
             (-\Delta)^{-\frac d4}T+R.
\]
Bounded left and right multiplication preserve
$(\mathcal L_{1,\infty}(H))_0$ by
\eqref{eq:prelim-ideal-property} and its definition as the
$\mathcal L_{1,\infty}$-closure of the finite-rank operators.
Thus \eqref{eq:ow-riesz-bessel-difference} gives
\[
\begin{aligned}
 &K'-S(1-\Delta)^{-\frac d4}\lambda(x)
                         (1-\Delta)^{-\frac d4}T\\
 =&R+S\Bigl[
       (-\Delta)^{-\frac d4}\lambda(x)(-\Delta)^{-\frac d4}
       -(1-\Delta)^{-\frac d4}\lambda(x)
                         (1-\Delta)^{-\frac d4}
       \Bigr]T\\
 \in &(\mathcal L_{1,\infty}(H))_0.
\end{aligned}
\]
The generalized principal symbol of $ K'$ is still $a$.
The Bessel-potential case gives \eqref{eq:ow-sandwich-singular},
\eqref{eq:ow-sandwich-signed}, and uniqueness of the generalized
principal symbol.
\end{proof}

The homogeneous extension of the generalized principal symbol is
\begin{equation}
\label{eq:ow-generalized-symbol-extension}
 \xi\longmapsto|\xi|^{-d}a(\tfrac{\xi}{|\xi|}),\qquad \xi\ne0.
\end{equation}
For a classical pseudodifferential operator $K\in\mathrm\Psi_{\mathrm{cl}}^{-d}(\mathbb T_\theta^d)$,
take $S=K(1-\Delta)^{d/2}$, $T=1$, $x=1$ and $R=0$
in \eqref{eq:ow-operator-and-symbol}.
Equation~\eqref{eq:prelim-classical-factorization} gives
$a(s)=\sigma(K)_{-d}(s)$.
Thus the generalized principal symbol agrees with the classical principal symbol
restricted to the sphere.

\begin{corollary}
\label{cor:ow-classical-factors}
Let $P,Q\in\mathrm \Psi_{\mathrm{cl}}^{-\frac{d}{2}}(\mathbb T_\theta^d)$ and $x\in L\log L(\mathbb T_\theta^d)$.
For unbounded $x$, let $Q\lambda(x)P$ denote the unique bounded
operator representing the form
\begin{equation}
\label{eq:ow-classical-form}
 (\xi,\eta)\longmapsto\tau\bigl((Q^*\eta)^*xP\xi\bigr),
 \qquad \xi,\eta\in \mathcal{S}(\mathbb T_\theta^d).
\end{equation}
This operator exists and belongs to $\mathcal{L}_{1,\infty}(H)$.
Writing $p(s)=\sigma(P)_{-\frac{d}{2}}(s)$ and
$q(s)=\sigma(Q)_{-\frac{d}{2}}(s)$, we have
\begin{equation*}
 \lim_{j\to\infty}(j+1)\mu(j,Q\lambda(x)P)
 =\frac{1}{d}\int_{\mathbb S^{d-1}}\tau(|q(s)xp(s)|)\,ds.
\end{equation*}
If $x=x^*$, then
\begin{equation*}
 \lim_{j\to\infty}(j+1)\lambda_j^\pm(P^*\lambda(x)P)
 =\frac{1}{d}\int_{\mathbb S^{d-1}}
       \tau\bigl((p(s)^*xp(s))_\pm\bigr)\,ds.
\end{equation*}
\end{corollary}

\begin{proof}
By \eqref{eq:prelim-classical-factorization}, the bounded extensions of
\[
 S=Q(1-\Delta)^{\frac{d}{4}},\qquad T=(1-\Delta)^{\frac{d}{4}}P
\]
belong to $\Pi$, with $\operatorname{sym}(S)=q$ and
$\operatorname{sym}(T)=p$. They satisfy $Q=S(1-\Delta)^{-\frac{d}{4}}$ and $P=(1-\Delta)^{-\frac{d}{4}}T$.
Since $x\in L_1(\mathbb T_\theta^d)$,
\[
|\tau((Q^*\eta)^*xP\xi)|
\leq\|x\|_1\|Q^*\eta\|_\infty\|P\xi\|_\infty.
\]

Since $P,Q^*$ act continuously on $\mathcal S(\mathbb T_\theta^d)$,
the form in \eqref{eq:ow-classical-form} is continuous for
the topology defined by \eqref{eq:prelim-smooth-seminorms}.
By \eqref{eq:prelim-smooth-fourier-density}, the identity
\eqref{eq:torus-sandwich-form} extends from Fourier polynomials to
$\mathcal S(\mathbb T_\theta^d)$.
The operators $S^*,T$ preserve $\mathcal S(\mathbb T_\theta^d)$. Hence
\[
 S(1-\Delta)^{-\frac d4}\lambda(x)(1-\Delta)^{-\frac d4}T
\]
represents the form in \eqref{eq:ow-classical-form}.
Theorem~\ref{thm:ow-sandwich-weyl} with $R=0$ gives
the singular-value asymptotics.
When $Q=P^*$, we have $S=T^*$, and the eigenvalue formulas
in Theorem~\ref{thm:ow-sandwich-weyl} apply.
\end{proof}

\subsection{Connes' integration formulas}

Fix an orthonormal basis $\{e_j\}_{j\geq0}$ of $H$ and define
$T_0e_j=\frac{1}{j+1}e_j$.
A trace on $\mathcal{L}_{1,\infty}(H)$ is a linear functional $\varphi$
satisfying $\varphi(AB)=\varphi(BA)$ for
$A\in\mathcal{L}_{1,\infty}(H)$ and $B\in\mathcal B(H)$.
It is positive if $\varphi(A)\geq0$ for $A\geq0$, and normalized if
$\varphi(T_0)=1$; see \cite[Section~4, before Corollary~4.2]{SSUZ15}.
We call $\varphi$ continuous if there is a constant $C_\varphi$ such that
\begin{equation}
\label{eq:ow-trace-continuity}
 |\varphi(T)|\leq C_\varphi\|T\|_{1,\infty},
 \qquad T\in\mathcal L_{1,\infty}(H).
\end{equation}

We use the following definition from
\cite[Definition~7.1]{SSUZ15}.

\begin{definition}
\label{def:ow-pt-measurable}
\normalfont
An operator $T\in\mathcal L_{1,\infty}(H)$ is
$\mathcal{PT}$-measurable if there is $c_T\in\mathbb C$
such that
\[
 \varphi(T)=c_T
\]
for every positive normalized trace $\varphi$ on
$\mathcal L_{1,\infty}(H)$.
\end{definition}

This is also called strong measurability;
see \cite[Definition~4.1]{P23}.

\begin{lemma}
\label{lem:ow-continuous-traces}
Every positive normalized trace on $\mathcal L_{1,\infty}(H)$ is
continuous. Every continuous trace vanishes on
$(\mathcal L_{1,\infty}(H))_0$.
If $B\in\mathcal L_{1,\infty}(H)$ is positive and
$(j+1)\mu(j,B)\to c$, then every continuous normalized trace
$\varphi$ satisfies
\begin{equation}
\label{eq:ow-trace-from-limit}
 \varphi(B)=c.
\end{equation}
\end{lemma}
\begin{proof}
The first assertion follows from
\cite[Theorem~1.1.12(g)]{LMSZ23}.
By \cite[Theorem~1.1.12(b)]{LMSZ23}, every trace vanishes
on finite-rank operators. Continuity therefore gives
\begin{equation}
\label{eq:ow-trace-small-ideal}
 \varphi(R)=0,
 \qquad R\in(\mathcal L_{1,\infty}(H))_0.
\end{equation}

For positive $B$, put $D_Be_j=\mu(j,B)e_j$.
By \cite[Theorem~1.1.12(a)]{LMSZ23},
\begin{equation}
\label{eq:ow-trace-diagonalization}
 \varphi(B)=\varphi(D_B).
\end{equation}
Let $P_N$ project onto
$\operatorname{span}\{e_0,\ldots,e_N\}$.
Since $(j+1)\mu(j,B)\to c$,
\begin{equation}
\label{eq:ow-trace-diagonal-tail}
 \|(D_B-cT_0)(I-P_N)\|_{1,\infty}
 \leq\sup_{j>N}|(j+1)\mu(j,B)-c|
 \longrightarrow0.
\end{equation}
Thus $D_B-cT_0\in(\mathcal L_{1,\infty}(H))_0$.
Equations~\eqref{eq:ow-trace-small-ideal} and
\eqref{eq:ow-trace-diagonalization} give
$\varphi(B)=c\varphi(T_0)=c$.
\end{proof}

\begin{theorem}
\label{thm:ow-connes}
Let $K$ and $a$ be as in Theorem~\ref{thm:ow-sandwich-weyl}.
Then $K$ and $|K|$ are $\mathcal{P T}$-measurable.
Every continuous normalized trace $\varphi$ on
$\mathcal L_{1,\infty}(H)$ satisfies
\begin{align*}
 \varphi(K)
 &=\frac{1}{d}\int_{\mathbb S^{d-1}}\tau(a(s))\,ds,
 \\
 \varphi(|K|)
 &=\frac{1}{d}\int_{\mathbb S^{d-1}}\tau(|a(s)|)\,ds.
\end{align*}
In particular, for the classical pseudodifferential operators in
Corollary~\ref{cor:ow-classical-factors},
\[
 \varphi(Q\lambda(x)P)
 =\frac{1}{d}\int_{\mathbb S^{d-1}}\tau(q(s)xp(s))\,ds.
\]
\end{theorem}

\begin{proof}
Let $\varphi$ be a continuous normalized trace.
We first prove the formula for $\varphi(K)$.
Equation~\eqref{eq:ow-riesz-bessel-difference} and
Lemma~\ref{lem:ow-continuous-traces} reduce the Riesz-potential case
to the Bessel-potential case.
Let $K_n,A_n,a_n$ be the approximants constructed in the proof of
Theorem~\ref{thm:ow-sandwich-weyl}.
Equation~\eqref{eq:ow-normal-form}, the commutator inclusion in
Proposition~\ref{prop:ow-bounded-input}, and
\eqref{eq:prelim-small-product} give
\[
 K_n-(1-\Delta)^{-\frac d4}A_n(1-\Delta)^{-\frac d4}
 \in(\mathcal L_{1,\infty}(H))_0.
\]
Both $\operatorname{Re}A_n$ and $\operatorname{Im}A_n$ belong to $\Pi$.
Apply the eigenvalue formulas in
Proposition~\ref{prop:ow-bounded-input} and
Lemma~\ref{lem:ow-continuous-traces} to the positive and negative
parts of their selfadjoint sandwiches.
By linearity and the vanishing of $\varphi$ on
$(\mathcal L_{1,\infty}(H))_0$, we obtain
\[
 \varphi(K_n)
 =\varphi\bigl((1-\Delta)^{-\frac d4}A_n(1-\Delta)^{-\frac d4}\bigr)
 =\frac1d\int_{\mathbb S^{d-1}}\tau(a_n(s))\,ds.
\]
Since $\|K_n-K\|_{1,\infty}\to0$, continuity of $\varphi$ and
\eqref{eq:ow-symbol-convergence} give the formula for $\varphi(K)$.

Since $\mu(j,|K|)=\mu(j,K)$,
Theorem~\ref{thm:ow-sandwich-weyl} and
Lemma~\ref{lem:ow-continuous-traces} give the formula for
$\varphi(|K|)$.
Every positive normalized trace is continuous by the same lemma.
Thus both formulas hold for all positive normalized traces,
which proves the $\mathcal{P T}$-measurability of $K$ and $|K|$.
\end{proof}

\begin{corollary}
\label{cor:ow-integration-potential}
For every $x\in L\log L(\mathbb T_\theta^d)$,
\begin{equation*}
 \lim_{j\to\infty}(j+1)\mu\bigl(j,(1-\Delta)^{-\frac{d}{4}}\lambda(x)(1-\Delta)^{-\frac{d}{4}}\bigr)
 =\frac{|\mathbb S^{d-1}|}{d}\tau(|x|),
\end{equation*}
and every continuous normalized trace $\varphi$ on
$\mathcal L_{1,\infty}(H)$ satisfies
\begin{equation*}
 \begin{aligned}
 \varphi\bigl((1-\Delta)^{-\frac{d}{4}}\lambda(x)(1-\Delta)^{-\frac{d}{4}}\bigr)
 &=\frac{|\mathbb S^{d-1}|}{d}\tau(x),\\
 \varphi\bigl(|(1-\Delta)^{-\frac{d}{4}}\lambda(x)(1-\Delta)^{-\frac{d}{4}}|\bigr)
 &=\frac{|\mathbb S^{d-1}|}{d}\tau(|x|).
 \end{aligned}
\end{equation*}
For $x=x^*$,
\begin{equation*}
 \lim_{j\to\infty}(j+1)\lambda_j^\pm\bigl((1-\Delta)^{-\frac{d}{4}}\lambda(x)(1-\Delta)^{-\frac{d}{4}}\bigr)
 =\frac{|\mathbb S^{d-1}|}{d}\tau(x_\pm),
\end{equation*}
and consequently
\begin{equation}
\label{eq:ow-constructive-integration}
 \begin{aligned}
 \tau(x)=\frac{d}{|\mathbb S^{d-1}|}\Bigl[
 &\lim_{j\to\infty}(j+1)\lambda_j^+\bigl((1-\Delta)^{-\frac{d}{4}}\lambda(x)(1-\Delta)^{-\frac{d}{4}}\bigr)\\
 &-\lim_{j\to\infty}(j+1)\lambda_j^-\bigl((1-\Delta)^{-\frac{d}{4}}\lambda(x)(1-\Delta)^{-\frac{d}{4}}\bigr)\Bigr].
 \end{aligned}
\end{equation}
The corresponding logarithmic mean also converges:
\begin{equation*}
 \begin{aligned}
 \tau(x)
 &=\frac{d}{|\mathbb S^{d-1}|}
   \lim_{N\to\infty}\frac{1}{\log(N+2)}\sum_{j=0}^N
   \Bigl[\lambda_j^+\bigl((1-\Delta)^{-\frac{d}{4}}\lambda(x)(1-\Delta)^{-\frac{d}{4}}\bigr)\\
 &\hspace{11em}
         -\lambda_j^-\bigl((1-\Delta)^{-\frac{d}{4}}\lambda(x)(1-\Delta)^{-\frac{d}{4}}\bigr)\Bigr].
 \end{aligned}
\end{equation*}
All formulas in this corollary also hold with $(-\Delta)^{-\frac{d}{4}}$ in place
of each $(1-\Delta)^{-\frac{d}{4}}$.
\end{corollary}

\begin{proof}
In \eqref{eq:ow-operator-and-symbol}, take $S=T=1$ and $R=0$.
The sandwich with $S=T=1$ has generalized principal symbol $a(s)=x$.
Apply Theorems~\ref{thm:ow-sandwich-weyl} and~\ref{thm:ow-connes}.
For the logarithmic-mean formula, use the elementary implication
\[
 (j+1)b_j\to c
 \quad\Longrightarrow\quad
 \frac{1}{\log(N+2)}\sum_{j=0}^N b_j\to c
\]
for each of the two eigenvalue sequences.
For the Riesz-potential formulas, use the Riesz-potential version
of Theorem~\ref{thm:ow-sandwich-weyl} and Theorem~\ref{thm:ow-connes}.
\end{proof}

Equation~\eqref{eq:ow-constructive-integration} gives $\tau(x)$ from the
positive and negative eigenvalue limits.
The corresponding result for compact Riemannian manifolds is
\cite[Theorem~1.2]{SZ23}.
For a general $x\in L\log L(\mathbb T_\theta^d)$, apply
\eqref{eq:ow-constructive-integration} to $\operatorname{Re}x$
and $\operatorname{Im}x$ separately.
\subsection{A semiclassical Weyl's law}

\begin{corollary}
\label{cor:ow-semiclassical}
Let $x=x^*\in L\log L(\mathbb T_\theta^d)$ and $h>0$. The form sum
\[
 H_h=h^d(1-\Delta)^{\frac{d}{2}}+\lambda(x)
\]
with form domain $W_2^{d/2}(\mathbb T_\theta^d)$ defines a selfadjoint operator
bounded below with compact resolvent. If $N_-(H_h)$ counts its strictly
negative eigenvalues with multiplicities, then
\begin{equation*}
 \lim_{h\downarrow0}h^dN_-(H_h)
 =\frac{|\mathbb S^{d-1}|}{d}\tau(x_-).
\end{equation*}
Moreover,
\begin{equation*}
 N_-(H_h)\leq \frac{C_d}{h^d}\|x_-\|_{L\log L(\mathbb T_\theta^d)},
\end{equation*}
with the constant from Theorem~\ref{thm:torus-endpoint}.
\end{corollary}

\begin{proof}
The identity~\eqref{eq:torus-sandwich-form}, the estimate \eqref{eq:torus-endpoint}, and the isomorphism in
\eqref{eq:prelim-sobolev} extend the potential form
$u\mapsto\tau(u^*xu)$ continuously from $\mathcal S(\mathbb T_\theta^d)$
to $W_2^{d/2}(\mathbb T_\theta^d)$.
We use $\tau(u^*xu)$ for this extension.
For every $\varepsilon>0$, Lemma~\ref{lem:ow-smooth-density} provides
bounded selfadjoint $y$ with $C_d\|x-y\|_{L\log L(\mathbb T_\theta^d)}<\varepsilon$.
The endpoint estimate then gives
\begin{equation}
\label{eq:ow-infinitesimal-form-bound}
 |\tau(u^*xu)|
 \leq\varepsilon\|(1-\Delta)^{\frac{d}{4}}u\|_2^2
       +\|y\|_\infty\|u\|_2^2.
\end{equation}

Define
\begin{equation}
\label{eq:ow-semiclassical-form}
 \mathfrak q_h[u]=h^d\|(1-\Delta)^{\frac d4}u\|_2^2+\tau(u^*xu),
 \qquad u\in W_2^{d/2}(\mathbb T_\theta^d).
\end{equation}
This form is densely defined, closed and bounded below.
Indeed, take $\varepsilon<h^d$ in
\eqref{eq:ow-infinitesimal-form-bound}.
For a sufficiently large $c$, the norm
$(\mathfrak q_h[u]+c\|u\|_2^2)^{1/2}$ is equivalent to the
$W_2^{d/2}(\mathbb T_\theta^d)$ norm; see
\cite[Chapter~VI, Theorem~1.33]{Kato95}.
The first representation theorem
\cite[Chapter~VI, Theorem~2.1]{Kato95} defines $H_h$.
Since the form is symmetric, $H_h$ is selfadjoint and bounded below
by \cite[Chapter~VI, Theorem~2.6]{Kato95}.
The compact embedding
$W_2^{d/2}(\mathbb T_\theta^d)\hookrightarrow H$ gives compact resolvent.

By \eqref{eq:prelim-sobolev}, $(1-\Delta)^{-d/4}$ is a bijection
from $H$ onto $W_2^{d/2}(\mathbb T_\theta^d)$.
Equations~\eqref{eq:torus-sandwich-form} and
\eqref{eq:ow-semiclassical-form} give
\begin{equation}
\label{eq:ow-transformed-form}
 \mathfrak q_h[(1-\Delta)^{-\frac d4}\xi]
 =h^d\|\xi\|_2^2
+\big\langle(1-\Delta)^{-\frac d4}\lambda(x)
                     (1-\Delta)^{-\frac d4}\xi,\xi\big\rangle_H,
 \qquad \xi\in H.
\end{equation}
This change of variables preserves the dimensions of subspaces
on which the forms are strictly negative.
Applying Glazman's variational principle
 (see, e.g., \cite[Lemma~7.1]{MP22}) to \eqref{eq:ow-transformed-form} gives
\begin{equation}
\label{eq:ow-bs-identity}
 N_-(H_h)
 =\#\bigl\{j:\lambda_j^-\bigl((1-\Delta)^{-\frac{d}{4}}\lambda(x)(1-\Delta)^{-\frac{d}{4}}\bigr)>h^d\bigr\}.
\end{equation}
The eigenvalue asymptotics in
Corollary~\ref{cor:ow-integration-potential},
the counting identity~\eqref{eq:ow-bs-identity}, and
Lemma~\ref{lem:distribution-singular-asymptotics}
give the asserted limit as $h\downarrow0$,
including the case $\tau(x_-)=0$.

Finally, $x=x_+-x_-$ with $x_\pm\geq0$.
Positivity of the potential forms and the min--max principle
\cite[Lemma~7.1]{MP22} yield
\[
 \begin{aligned}
 N_-(H_h)
 &\leq\#\bigl\{j:\mu\bigl(j,(1-\Delta)^{-\frac{d}{4}}\lambda(x_-)(1-\Delta)^{-\frac{d}{4}}\bigr)>h^d\bigr\}\\
 &\leq\frac{1}{h^d}\bigl\|(1-\Delta)^{-\frac{d}{4}}\lambda(x_-)(1-\Delta)^{-\frac{d}{4}}\bigr\|_{1,\infty}\\
 &\leq\frac{C_d}{h^d}\|x_-\|_{L\log L(\mathbb T_\theta^d)}.
 \end{aligned}
\]
This proves the bound for $N_-(H_h)$.
\end{proof}

\medskip

\begingroup
\renewcommand{\addcontentsline}[3]{}

\subsection*{Acknowledgements}
The authors would like to thank Fedor Sukochev and Dmitriy Zanin for their interest in this work, helpful comments, and suggestions that improved the presentation of the proof. 
The authors are also grateful to Edward McDonald and Xiao Xiong for helpful discussions.

\subsection*{AI statement}
The authors use ChatGPT for language polishing, LaTeX editing and exploratory mathematical discussions. 
GPT-5.6 Sol suggested the approach to the estimate for the bounded cutoff via a convolution argument (see equation \eqref{eq:cwikel-cutoff-profile} and the paragraph it lies in), which is inspired by the paper \cite{HKRV23}. 
The authors checked the suggested idea and the existence of \cite{HKRV23}, and take full responsibility for the mathematical content and conclusions of this work. 
\par
\endgroup

\bibliographystyle{plain}
\bibliography{references}

\end{document}